\documentclass[11pt, reqno]{amsart}
\usepackage{amssymb,latexsym,amsmath,amsfonts,mathdots,enumitem}
\usepackage{latexsym}
\usepackage[mathscr]{eucal}
\usepackage{colortbl,xcolor}
\usepackage{lmodern}
\usepackage{sansmathaccent}
\usepackage[latin1]{inputenc}
\usepackage{tikz}
\usepackage{physics}
\usepackage{tikz-cd}
\usetikzlibrary{shapes,arrows}
\usetikzlibrary{matrix,calc,shapes,arrows,positioning}
\pdfmapfile{+sansmathaccent.map}

\numberwithin{equation}{section}
\theoremstyle{plain}
\newtheorem{exam}{Example}

\newtheorem{thm}{Theorem}[section]

\newtheorem{cor}[thm]{Corollary}
\newtheorem{lem}[thm]{Lemma}
\newtheorem{prop}[thm]{Proposition}
\theoremstyle{definition}
\newtheorem{defn}[thm]{Definition}

\numberwithin{equation}{section}
\def\A{{\mathcal A}}

\def\Hol{\operatorname{Hol}}

\def\beq{\begin{eqnarray}}
	\def\eeq{\end{eqnarray}}
\def\beqa{\begin{eqnarray*}}
	\def\eeqa{\end{eqnarray*}}

\def\Ran{\operatorname{Ran}}

\def\beqn{\begin{equation}}
	\def\eeqn{\end{equation}}

\def\mg#1{}

\def\Ran{\operatorname{Ran}}

\renewcommand{\epsilon}{\varepsilon}
\renewcommand{\phi}{\varphi}

\renewcommand{\bf}[1]{\textbf{#1}}
\renewcommand{\it}[1]{\textit{#1}}

\renewcommand{\sf}[1]{\textsf{#1}}

\numberwithin{equation}{section}
\allowdisplaybreaks[4] 

\setlist[enumerate]{font=\upshape,noitemsep, topsep=0pt} 
\setlist[itemize]{noitemsep, topsep=0pt}

\begin{document}
	
	\title[Functional Models for $\Gamma_{E(3; 3; 1, 1, 1)}$, $\Gamma_{E(3; 2; 1, 2)}$ and Tetrablock contraction]{Functional Models for $\Gamma_{E(3; 3; 1, 1, 1)}$-contraction, $\Gamma_{E(3; 2; 1, 2)}$-contraction and Tetrablock contraction}
	\author{Dinesh Kumar Keshari, Suryanarayan Nayak, Avijit Pal  and Bhaskar Paul}
	\address[ D. K. Keshari]{School of Mathematical Sciences, National Institute of Science Education and Research Bhubaneswar, An OCC of Homi Bhabha National Institute, Jatni, Khurda,  Odisha-752050, India}
\email{dinesh@niser.ac.in}

\address[S. Nayak]{School of Mathematical Sciences, National Institute of Science Education and Research Bhubaneswar, An OCC of Homi Bhabha National Institute, Jatni, Khurda,  Odisha-752050, India}
\email{suryanarayan.nayak@niser.ac.in}

\address[A. Pal]{Department of Mathematics, IIT Bhilai, 6th Lane Road, Jevra, Chhattisgarh 491002}
\email{A. Pal:avijit@iitbhilai.ac.in}

\address[B. Paul]{Department of Mathematics, IIT Bhilai, 6th Lane Road, Jevra, Chhattisgarh 491002}
\email{B. Paul:bhaskarpaul@iitbhilai.ac.in }


	\maketitle

	\maketitle
	
	\begin{abstract}
Let $(A,B,P)$ be a commuting triple of bounded operators on a Hilbert space $\mathcal H.$ We say that $(A,B,P)$  is a tetrablock contraction if $\Gamma_{E(2;2;1,1)}$ is a spectral set for $(A,B,P).$ 
 If $\Gamma_{E(3; 3; 1, 1, 1)}$ is a spectral set for $\textbf{T} = (T_1, \dots, T_7)$, then a $7$-tuple of commuting bounded operators $\textbf{T}$ on some Hilbert space $\mathcal{H}$ is referred to as a \textit{$\Gamma_{E(3; 3; 1, 1, 1)}$-contraction}.  Let $(S_1, S_2, S_3)$ and $(\tilde{S}_1, \tilde{S}_2)$ be tuples of commuting bounded operators on some Hilbert space $\mathcal{H}$ with $S_i\tilde{S}_j = \tilde{S}_jS_i$ for $1 \leqslant i \leqslant 3$ and $1 \leqslant j \leqslant 2$. We say that  $\textbf{S} = (S_1, S_2, S_3, \tilde{S}_1, \tilde{S}_2)$ is a $\Gamma_{E(3; 2; 1, 2)}$-contraction if $ \Gamma_{E(3; 2; 1, 2)}$ is a spectral set for $\textbf{S}$. We obtain various characterizations of the fundamental operators of $\Gamma_{E(3; 3; 1, 1, 1)}$-contraction and $\Gamma_{E(3; 2; 1, 2)}$-contraction. We also demonstrate some important relations between the fundamental operators of a $\Gamma_{E(3; 3; 1, 1, 1)}$-contraction and a $\Gamma_{E(3; 2; 1, 2)}$-contraction. We describe functional models for \textit{pure $\Gamma_{E(3; 3; 1, 1, 1)}$-contraction} and \textit{pure $\Gamma_{E(3; 2; 1, 2)}$-contraction}.  We give a complete set of unitary invariants for a pure 	$\Gamma_{E(3; 3; 1, 1, 1)}$-contraction and a pure  $\Gamma_{E(3; 2; 1, 2)}$-contraction. We demonstrate the functional models for a certain class of completely non-unitary $\Gamma_{E(3; 3; 1, 1, 1)}$-contraction $\textbf{T} = (T_1, \dots, T_7)$ and completely non-unitary $\Gamma_{E(3; 2; 1, 2)}$-contraction $\textbf{S} = (S_1, S_2, S_3, \tilde{S}_1, \tilde{S}_2)$ which satisfy the following conditions:
	\begin{equation}\label{Condition 1}
		\begin{aligned}
			&T^*_iT_7 = T_7T^*_i \,\, \text{for} \,\, 1 \leqslant i \leqslant 6
		\end{aligned}
	\end{equation}
	and
	\begin{equation}\label{Condition 2}
		\begin{aligned}
			&S^*_iS_3 = S_3S^*_i, \tilde{S}^*_jS_3 = S_3\tilde{S}^*_j \,\, \text{for} \,\, 1 \leqslant i, j \leqslant 2,
		\end{aligned}
	\end{equation}
	respectively. We also describe  a functional model for a completely non-unitary tetrablock contraction $\textbf{T} = (A_1,A_2,P)$ that satisfies
	\begin{equation}\label{Condition 3}
		\begin{aligned}
			A^*_iP = PA^*_i \,\, \text{for $1 \leqslant i \leqslant 2$}.
		\end{aligned}
	\end{equation}
By exhibiting counter examples,   we show that such abstract model of tetrablock contraction, $\Gamma_{E(3; 3; 1, 1, 1)}$-contraction and $\Gamma_{E(3; 2; 1, 2)}$-contraction may not exist if we drop the hypothesis of  $(\ref{Condition 3})$ $(\ref{Condition 1})$, and $(\ref{Condition 2}),$ respectively.

	\end{abstract}
	
	\section{Introduction and Motivation}
Let $\mathbb C[z_1,\dots,z_n]$ denotes the polynomial ring in $n$ variables over the field of complex numbers. Let $\Omega$ be a compact subset of $\mathbb C^m,$ and let $\mathcal{O}(\Omega)$ denotes the algebra of holomorphic functions on an open set containing $\Omega.$ Let  $\mathbf{T}=(T_1,\ldots,T_m)$ be a commuting $m$-tuple of bounded operators defined on a  Hilbert space $\mathcal H$ and $\sigma(\mathbf T)$ denotes the joint spectrum of $\mathbf {T}.$  Consider the map $\rho_{\mathbf T}:\mathcal{O}(\Omega)\rightarrow\mathcal B(\mathcal H)$  defined by $$1\to I~{\rm{and}}~z_i\to T_i ~{\rm{for}}~1\leq i\leq m.$$ Clearly, $\rho_{\mathbf T}$ is a homomorphism. A compact set $\Omega\subset \mathbb C^m$ is a spectral set for a  $m$-tuple of commuting bounded operators $\mathbf{T}=(T_1,\ldots,T_m)$ if $\sigma(\mathbf T)\subseteq \Omega$ and the homomorphism $\rho_{\mathbf T}:\mathcal{O}(\Omega)\rightarrow\mathcal B(\mathcal H)$ is contractive.

Let $\mathcal M_{n\times n}(\mathbb{C})$ be the set of all $n\times n$ complex matrices and  $E$ be a linear subspace of $\mathcal M_{n\times n}(\mathbb{C}).$ We define the function $\mu_{E}: \mathcal M_{n\times n}(\mathbb{C}) \to [0,\infty)$ as follows:
\begin{equation}\label{mu}
\mu_{E}(A):=\frac{1}{\inf\{\|X\|: \,\ \det(1-AX)=0,\,\, X\in E\}},\;\; A\in \mathcal M_{n\times n}(\mathbb{C})
	\end{equation}
with the understanding that $\mu_{E}(A):=0$ if $1-AX$ is  nonsingular for all $X\in E$ \cite{ds, jcd}.   Here $\|\cdot\|$ denotes the operator norm. Let  $E(n;s;r_{1},\dots,r_{s})\subset \mathcal M_{n\times n}(\mathbb{C})$ be the vector subspace comprising block diagonal matrices, defined as follows:
\begin{equation}\label{ls}
    	E=E(n;s;r_{1},...,r_{s}):=\{\operatorname{diag}[z_{1}I_{r_{1}},....,z_{s}I_{r_{s}}]\in \mathcal M_{n\times n}(\mathbb{C}): z_{1},...,z_{s}\in \mathbb{C}\},
\end{equation}
 where $\sum_{i=1}^{s}r_i=n.$ We recall the definition of $\Gamma_{E(3; 3; 1, 1, 1)}$, $\Gamma_{E(3; 2; 1, 2)}$ and $\Gamma_{E(2; 2; 1, 1)}$ \cite{Abouhajar,Bharali, apal1}. The sets $\Gamma_{E{(2;2;1,1)}}$, $\Gamma_{E(3; 3; 1, 1, 1)}$ and $\Gamma_{E(3; 2; 1, 2)}$ are defined as 
 \begin{equation*}
\begin{aligned}
\Gamma_{E{(2;2;1,1)}}:=\Big \{\textbf{x}=(x_1=a_{11}, x_2=a_{22}, x_3=a_{11}a_{22}-a_{12}a_{21}=\det A)\in \mathbb C^3: &\\A\in \mathcal M_{2\times 2}(\mathbb C)~{\rm{and}}~\mu_{E(2;2;1,1)}(A)\leq 1\Big \},
\end{aligned}
\end{equation*}	 
 \begin{equation*}
\begin{aligned}
\Gamma_{E{(3;3;1,1,1)}}:=\Big \{\textbf{x}=(x_1=a_{11}, x_2=a_{22}, x_3=a_{11}a_{22}-a_{12}a_{21}, x_4=a_{33}, x_5=a_{11}a_{33}-a_{13}a_{31},&\\ x_6=a_{22}a_{33}-a_{23}a_{32},x_7=\det A)\in \mathbb C^7: A\in \mathcal M_{3\times 3}(\mathbb C)~{\rm{and}}~\mu_{E(3;3;1,1,1)}(A)\leq 1\Big \}
\end{aligned}
\end{equation*}	
$${\rm{and}}$$  
\begin{equation*}
\begin{aligned}
\Gamma_{E(3;2;1,2)}:=\Big\{( x_1=a_{11},x_2=\det \left(\begin{smallmatrix} a_{11} & a_{12}\\
					a_{21} & a_{22}
				\end{smallmatrix}\right)+\det \left(\begin{smallmatrix}
					a_{11} & a_{13}\\
					a_{31} & a_{33}
				\end{smallmatrix}\right),x_3=\operatorname{det}A, y_1=a_{22}+a_{33}, &\\ y_2=\det  \left(\begin{smallmatrix}
					a_{22} & a_{23}\\
					a_{32} & a_{33}\end{smallmatrix})\right)\in \mathbb C^5
:A\in \mathcal M_{3\times 3}(\mathbb C)~{\rm{and}}~\mu_{E(3;2;1,2)}(A)\leq 1\Big\}.
\end{aligned}
\end{equation*}

The sets $\Gamma_{E(3; 2; 1, 2)}$ and $\Gamma_{E(2; 2; 1, 1)}$  are referred to as $\mu_{1,3}-$\textit{quotient} and tetrablock, respectively \cite{Abouhajar, Bharali}. 

Let $T$ be a contraction on a hilbert space $\mathcal H$ is called pure if $T^{n*}\to 0$ strongly, that is, $\|T^{n*}h\|\to o,$ for all $h\in \mathcal H.$
\begin{defn}\label{Defn 1}
		\begin{enumerate}
		\item Let $(A,B,P)$ be a commuting triple of bounded operators on a Hilbert space $\mathcal H.$ We say that $(A,B,P)$  is a tetrablock contraction if $\Gamma_{E(2;2;1,1)}$ is a spectral set for $(A,B,P).$ 
		\item A tetrablock contraction $(A,B,P)$ is pure if the contraction $P$ is pure.
			
			\item If $\Gamma_{E(3; 3; 1, 1, 1)}$ is a spectral set for $\textbf{T} = (T_1, \dots, T_7)$, then a $7$-tuple of commuting bounded operators $\textbf{T}$ on some Hilbert space $\mathcal{H}$ is referred to as a \textit{$\Gamma_{E(3; 3; 1, 1, 1)}$-contraction}.
			
			\item A  \textit{$\Gamma_{E(3; 3; 1, 1, 1)}$-contraction}  $\textbf{T} = (T_1, \dots, T_7)$ is called pure if the contraction $T_7$ is pure.
			
			\item Let $(S_1, S_2, S_3)$ and $(\tilde{S}_1, \tilde{S}_2)$ be tuples of commuting bounded operators on some Hilbert space $\mathcal{H}$ with $S_i\tilde{S}_j = \tilde{S}_jS_i$ for $1 \leqslant i \leqslant 3$ and $1 \leqslant j \leqslant 2$. We say that  $\textbf{S} = (S_1, S_2, S_3, \tilde{S}_1, \tilde{S}_2)$ is a $\Gamma_{E(3; 2; 1, 2)}$-contraction if $ \Gamma_{E(3; 2; 1, 2)}$ is a spectral set for $\textbf{S}$.
			\item A $\Gamma_{E(3; 2; 1, 2)}$-contraction is called pure if $S_3$ is a pure contraction.
			
					\end{enumerate}
	\end{defn}
Let $T$ be a contraction on a Hilbert space $\mathcal H.$ Define the defect operator $D_{T}=(I-T^*T)^{\frac{1}{2}}$ associated with $T$. The closure of the range of $ D_{T}$ is denoted by $\mathcal D_{T}$.  	
\begin{defn}\label{fundamental}
		Let $(T_1, \dots, T_7)$ be a $7$-tuple of commuting contractions on a Hilbert space $\mathcal{H}. $ The equations 

		\begin{equation}\label{Fundamental 1}
\begin{aligned}
&T_i - T^*_{7-i} T_7 = D_{T_7}F_iD_{T_7}, \;\;\; 1\leq i\leq 6, 
\end{aligned}
\end{equation}
where $F_i\in \mathcal{B}(\mathcal{D}_{T_7}),$ are referred to as the  fundamental equations for $(T_1, \dots, T_7)$.
			\end{defn}
	
\begin{defn}\label{fundamental}
Let $(S_1, S_2, S_3, \tilde{S}_1, \tilde{S}_2)$ be a $5$-tuple of commuting bounded operators defined on a Hilbert space $\mathcal H$. The equations
\begin{equation}
			\begin{aligned}\label{funda1}
					&S_1 - \tilde{S}^*_2S_3 = D_{S_3}G_1D_{S_3},\,\, \tilde{S}_2 - S^*_1S_3 = D_{S_3}\tilde{G}_2D_{S_3},
				\end{aligned}
			\end{equation}
			$${\rm{and}}$$
		\begin{equation}
				\begin{aligned}\label{funda11}
				&\frac{S_2}{2} - \frac{\tilde{S}^*_1}{2}S_3 = D_{S_3}G_2D_{S_3}, \,\, \frac{\tilde{S}_1}{2} - \frac{S^*_2}{2}S_3 = D_{S_3}\tilde{G}_1D_{S_3},				\end{aligned}
			\end{equation}
where $G_1,2G_2,2\tilde{G}_1$ and $\tilde{G}_2$ in $\mathcal{B}(\mathcal{D}_{S_3}),$ are referred to as the  fundamental equations for $(S_1, S_2, S_3, \tilde{S}_1, \tilde{S}_2)$.		\end{defn}
We denote the unit circle by $\mathbb T.$  Let $\mathcal E$  be a separable Hilbert space. Let  $\mathcal B(\mathcal E)$ denote the space of bounded linear operators on $\mathcal E$ equipped with the operator norm. Let $H^2(\mathcal E)$ denote  the  Hardy space of analytic $\mathcal E$-valued functions defined on the unit disk  $\mathbb D$. Let $ L^2(\mathcal E)$ represent the Hilbert space of square-integrable $\mathcal E$-valued functions on the unit circle $\mathbb T,$ equipped with the natural inner product. The space $H^{\infty}(\mathcal B(\mathcal E))$ consists of bounded analytic $\mathcal B(\mathcal E)$-valued functions defined on $\mathbb D$. Let $L^{\infty}(\mathcal B(\mathcal E))$ denote the space of bounded measurable $\mathcal B(\mathcal E)$-valued functions on $\mathbb T$.  For $\phi \in L^{\infty}(\mathcal B(\mathcal E)),$ the Toeplitz operator associated with the symbol  $\phi$ is denoted by $T_{\phi}$ and is defined as follows: 
$$T_{\phi}f=P_{+}(\phi f), f \in H^2(\mathcal E),$$ where $P_{+} : L^2(\mathcal E) \to H^2(\mathcal E)$ is the orthogonal projecton.  In particular, $T_z$ is the
unilateral shift operator $M_z$ on $H^2(\mathcal E)$  and $T_{\bar{z}}$ is the backward shift $M_z^*$ on $H^2(\mathcal E)$.
 The vector valued Hardy space is denoted by $H^2_{\mathcal{E}}(\mathbb{D})$. The space $H^2_{\mathcal{E}}(\mathbb{D})$ is unitarily equivalent to $H^2(\mathbb{D}) \otimes \mathcal{E}$ by the map $z^n\eta \mapsto z^n \otimes \eta$. Throughout this article we use the notation $H^2(\mathbb{D}) \otimes \mathcal{E}$.
	
Sz.-Nagy and Foias demonstrated a functional model for a pure contraction \cite{Nagy}. We first recall a little bit about the development. Let $T$ be a contraction a Hilbert space $\mathcal{H}$. Then the $D_T$ and $D_{T^*}$ satisfy the following identity:
	\begin{equation*}
		\begin{aligned}
			TD_T &= D_{T^*}T \hspace{0.5cm} \text{equivalently} \hspace{0.5cm} D_TT^* = T^*D_{T^*}.
		\end{aligned}
	\end{equation*}	
and its coresponding adjoint is given by\begin{equation*}
		\begin{aligned}
			 D_TT^* = T^*D_{T^*}.
		\end{aligned}
	\end{equation*}	
	The \textit{characteristic function} $\Theta_T$ of $T$ is defined as
	\begin{equation}\label{Characteristic}
		\begin{aligned}
			\Theta_T(z)
			&= (- T + D_{T^*}(I - zT^*)^{-1}D_T)_{|_{\mathcal{D}_T}}, \,\, \text{for all} \,\, z \in \mathbb{D}.
		\end{aligned}
	\end{equation}
	It is easy to notice that $\Theta \in \mathcal{B}(\mathcal{D}_T, \mathcal{D}_{T^*})$.  We define the  multiplication operator $M_{\Theta_T} : H^2(\mathbb{D}) \otimes \mathcal{D}_T \to H^2(\mathbb{D}) \otimes \mathcal{D}_{T^*}$ by
	\[M_{\Theta_T}f(z) = \Theta_T(z)f(z) \,\, \text{for} \,\, z \in \mathbb{D}.\]
Let  $\mathcal{H}_T = (H^2(\mathbb{D}) \otimes \mathcal{D}_{T^*}) \ominus M_{\Theta_T}(H^2(\mathbb{D}) \otimes \mathcal{D}_T)$. $\mathcal{H}_T$ is called the \textit{model space for $T$}. We now state the functional model for pure contraction from \cite{Nagy}.
	\begin{thm}\label{Pure Contraction Model}
		Every pure contraction $T$ defined on a Hilbert space $\mathcal{H}$ is unitarily 
		equivalent to the operator $T_1$ on the Hilbert space $\mathcal{H}_T = (H^2(\mathbb{D}) \otimes \mathcal{D}_{T^*}) \ominus M_{\Theta_T}(H^2(\mathbb{D}) \otimes \mathcal{D}_T)$ defined as
		\begin{equation}\label{Model}
			\begin{aligned}
				T_1 &= P_{\mathcal{H}_T}
				(M_z \otimes I_{\mathcal{D}_{T^*}})_{|_{\mathcal{H}_T}}.
			\end{aligned}
		\end{equation}
	\end{thm}

We recall the definition of completely non-unitary contraction from \cite{Nagy}. A contraction $T$  on a Hilbert space $\mathcal H$  is said to be  completely non-unitary (c.n.u.) contractions if there exists no nontrivial reducing subspace $\mathcal L$  for $T$ such that $T |_{\mathcal L}$ is a unitary operator. This section presents the canonical decomposition of the $\Gamma_{E(3; 3; 1, 1, 1)}$-contraction and the $\Gamma_{E(3; 2; 1, 2)}$-contraction.  Any contraction $T$ on a Hilbert space $\mathcal{H}$ can be expressed as the orthogonal direct sum of a unitary and a completely non-unitary contraction. The details can be found in [Theorem 3.2, \cite{Nagy}]. We start with the following definition, which will be essential for the canonical decomposition of the $\Gamma_{E(3; 3; 1, 1, 1)}$-contraction and the $\Gamma_{E(3; 2; 1, 2)}$-contraction.	
	\begin{defn}
		\begin{enumerate}
			\item A $\Gamma_{E(3; 3; 1, 1, 1)}$-contraction $\textbf{T} = (T_1, \dots, T_7)$ is said to be completely non-unitary  {$\Gamma_{E(3; 3; 1, 1, 1)}$-contraction} if $T_7$ is a completely non-unitary contraction.
			
			\item A $\Gamma_{E(3; 2; 1, 2)}$-contraction $\textbf{S} = (S_1, S_2, S_3, \tilde{S}_1, \tilde{S}_2)$ is said to be completely non-unitary {$\Gamma_{E(3; 2; 1, 2)}$-contraction} if $S_3$ is a completely non-unitary contraction.		\end{enumerate}
	\end{defn}

H. Sau \cite{Sau} produced a set of unitary invariants for pure tetrablock contraction $(A,B,P)$, which comprises three members: the characteristic function of $P$ and the two fundamental operators of $(A^*,B^*,P^*).$	T. Bhattacharyya, S. Lata and H. Sau \cite{TB} proved a set of unitary invariants for pure $\Gamma$-contraction. B. Bisai and S. Pal \cite{bisai1} extended the result for $\Gamma_n$-contraction. They also described the abstract model for a completely nonunitary $\Gamma_n$-contraction \cite{bisai2}.
	
In Section \ref{Section 2},  we obtain various characterizations of the fundamental operators of $\Gamma_{E(3; 3; 1, 1, 1)}$-contraction and $\Gamma_{E(3; 2; 1, 2)}$-contraction. We also demonstrate some important relations between the fundamental operators of a $\Gamma_{E(3; 3; 1, 1, 1)}$-contraction and a $\Gamma_{E(3; 2; 1, 2)}$-contraction. Section \ref{Section 3} is devoted to the main results of this article. We find functional models for \textit{pure $\Gamma_{E(3; 3; 1, 1, 1)}$-contraction} and \textit{pure $\Gamma_{E(3; 2; 1, 2)}$-contraction}.  In section $4$, we give a complete set of unitary invariants for a pure 	$\Gamma_{E(3; 3; 1, 1, 1)}$-contraction and a pure  $\Gamma_{E(3; 2; 1, 2)}$-contraction. In section $5$, we demonstrate the functional models for a certain class of completely non-unitary $\Gamma_{E(3; 3; 1, 1, 1)}$-contraction $\textbf{T} = (T_1, \dots, T_7)$ and completely non-unitary $\Gamma_{E(3; 2; 1, 2)}$-contraction $\textbf{S} = (S_1, S_2, S_3, \tilde{S}_1, \tilde{S}_2)$ which satisfy the  conditions \eqref{Condition 1} and \eqref{Condition 2}, respectively. We also describe  a functional model for a completely non-unitary tetrablock contraction $\textbf{R} = (R_1, R_2, R_3)$ that satisfies the condition \eqref{Condition 3}. In section $6$,  by exhibiting counter examples,   we show that such abstract model of tetrablock contraction, $\Gamma_{E(3; 3; 1, 1, 1)}$-contraction and $\Gamma_{E(3; 2; 1, 2)}$-contraction may not exist if we drop the hypothesis of  $(\ref{Condition 3})$ $(\ref{Condition 1})$, and $(\ref{Condition 2}),$ respectively.

	\section{Some Relations Among the Fundamental Operators}\label{Section 2}
In this section, we obtain various characterizations of the fundamental operators of $\Gamma_{E(3; 3; 1, 1, 1)}$-contraction and $\Gamma_{E(3; 2; 1, 2)}$-contraction. We also demonstrate some important relations between the fundamental operators of a $\Gamma_{E(3; 3; 1, 1, 1)}$-contraction and a $\Gamma_{E(3; 2; 1, 2)}$-contraction. 
\begin{prop}[ Proposition $2.11$, \cite{apal2}]\label{Prop 1}
		Let $(T_1, \dots, T_7)$ be a $\Gamma_{E(3; 3; 1, 1, 1)}$-contraction. Then $(T_1, T_6, T_7)$, $(T_2, T_5, T_7)$ and $(T_3, T_4, T_7)$ are 
		$\Gamma_{E(2; 2; 1, 1)}$-contractions.
	\end{prop}
	
\begin{prop}[Lemma $2.7$, \cite{apal3}]\label{FiFj}
	The fundamental operators  of a $\Gamma_{E(3; 3; 1, 1, 1)}$-contraction $\textbf{T} = (T_1, \dots, T_7)$ are the unique bounded linear operators $X_i$ and $X_{7-i}$, $1\leq i \leq 6$, defined on $\mathcal D_{T_7}$ satisfying the operator equations
\begin{equation}
			\begin{aligned}
				&D_{T_7}T_i = X_iD_{T_7} + X^*_{7-i}D_{T_7}T_7 ~\text{and}~ D_{T_7}T_{7-i} = X_{7-i}D_{T_7} + X^*_iD_{T_7}T_7~{\rm{for}}~1\leq i \leq 6.
			\end{aligned}
		\end{equation}
	\end{prop}
\begin{lem}[Lemma $2.8$, \cite{apal3}]\label{F12}
		Let $\textbf{T} = (T_1, \dots, T_7) $ be a $\Gamma_{E(3; 3; 1, 1, 1)} $-contraction on the Hilbert space $\mathcal{H}$ with commuting fundamental operators $F_i, 1\leq i \leq 6,$ defined on $\mathcal{D}_{T_7}. $ Then
		\begin{equation}\label{F_i}
			\begin{aligned}
				T_i^*T_i - T_{7-i}^*T_{7-i} = D_{T_7}(F^*_iF_i - F^*_{7-i}F_{7-i})D_{T_7},1\leq i \leq 6.
			\end{aligned}
		\end{equation}
	\end{lem}

\begin{prop}\label{Prop 3}
		Let $\textbf{T} = (T_1, \dots, T_7)$ be a $\Gamma_{E(3; 3; 1, 1, 1)}$-contraction on a Hilbert space $\mathcal{H}$. Suppose that $F_i,1\leq i \leq 6,$ are fundamental operators for $\textbf{T}$ and $\tilde{F}_j,1\leq j \leq 6,$ are fundamental operators for $\textbf{T}^* = (T^*_1, \dots, T^*_7)$. Then the following properties hold:
	\begin{enumerate}
\item $D_{T_7}F_i = (T_iD_{T_7} - D_{T^*_7}\tilde{F}_{7-i}T_7)|_{\mathcal{D}_{T_7}},~ 1\leq i \leq 6.$
			
\item $T_7F_i = \tilde{F}^*_iT_7|_{\mathcal{D}_{T_7}} \,\, \text{for} \,\, 1 \leqslant i \leqslant 6.$

\item $(F^*_iD_{T_7}D_{T^*_7} - F_{7-i}T^*_7)|_{\mathcal{D}_{T^*_7}}=
				D_{T_7}D_{T^*_7}\tilde{F}_i - T^*_7\tilde{F}^*_{7-i} \,\, \text{for} \,\, 1 \leqslant i \leqslant 6.$
				\end{enumerate}
	\end{prop}
	
	\begin{proof}
\begin{enumerate}

\item By Proposition \ref{Prop 1}, it follows that $(T_i, T_{7-i}, T_7), 1\leq i \leq 6,$ is a $\Gamma_{E(2; 2; 1,1)}$-contraction. Thus, $(T^*_i, T^*_{7-i}, T^*_7)$ is a $\Gamma_{E(2; 2; 1,1)}$-contraction for $1 \leqslant i \leqslant 6$ \cite{Bhattacharyya}.  For $h \in \mathcal{H}$, we note that
\begin{equation}\label{2.1}
			\begin{aligned}
				(T_iD_{T_7} - D_{T^*_7}\tilde{F}_{7-i}T_7)D_{T_7}h
				&=
				T_iD^2_{T_7}h - D_{T^*_7}\tilde{F}_{7-i}T_7D_{T_7}h\\
				&= T_i(I - T^*_7T_7)h - (D_{T^*_7}\tilde{F}_{7-i}D_{T^*_7})T_7h\\
				&= T_i(I - T^*_7T_7)h - (T^*_{7-i} - T_iT^*_7)T_7h\\
				&= (T_i - T^*_{7-i}T_7)h\\
				&= D_{T_7}F_iD_{T_7}h, 1\leq i \leq 6.
			\end{aligned}
		\end{equation}
From \eqref{2.1}, we deduce that $D_{T_7}F_i = (T_iD_{T_7} - D_{T^*_7}\tilde{F}_{7-i}T_7)|_{\mathcal{D}_{T_7}}$ for $1\leq i \leq 6.$ 

\item For $h_1, h_2 \in \mathcal{H}$, we have
		\begin{equation}\label{2.2}
			\begin{aligned}
				\langle(T_7F_i - \tilde{F}^*_iT_7)D_{T_7}h_1, D_{T^*_7}h_2\rangle
				&=
				\langle D_{T^*_7}T_7F_iD_{T_7}h_1, h_2 \rangle - \langle D_{T^*_7}\tilde{F}^*_iT_7D_{T_7}h_1, h_2 \rangle\\
				&= \langle T_7(D_{T_7}F_iD_{T_7})h_1, h_2 \rangle - \langle (D_{T^*_7}\tilde{F}^*_iD_{T^*_7})T_7h_1, h_2 \rangle\\
				&= \langle T_7(T_i - T^*_{7-i}T_7)h_1, h_2 \rangle - \langle (T^*_i - T_{7-i}T^*_7)^*T_7h_1, h_2 \rangle\\
				&= 0, 1\leq i \leq 6.
			\end{aligned}
		\end{equation}
		Therefore,  it follows from \eqref{2.2}  that $T_7F_i = \tilde{F}^*_iT_7|_{\mathcal{D}_{T_7}}$ for $1 \leqslant i \leqslant 6$.
		
\item For $h \in \mathcal{H}$, we observe that
		\begin{equation}\label{2.3}
			\begin{aligned}
				(F^*_iD_{T_7}D_{T^*_7} - F_{7-i}T^*_7)D_{T^*_7}h
				&= F^*_iD_{T_7}D^2_{T^*_7}h - F_{7-i}T^*_7D_{T^*_7}h\\
				&= F^*_iD_{T_7}(I - T_7T^*_7)h - F_{7-i}D_{T_7}T^*_7h\\
				&= F^*_iD_{T_7}h - (F^*_iD_{T_7}T_7 + F_{7-i}D_{T_7})T^*_7h\\
				&= F^*_iD_{T_7}h - D_{T_7}T_{7-i}T^*_7h \,\, (\text{by Proposition \ref{FiFj}})\\
				&= D_{T_7}T^*_ih - T^*_7\tilde{F}^*_{7-i}D_{T^*_7}h - D_{T_7}T_{7-i}T^*_7h \,\, (\text{by Part $(1)$})\\
				&= D_{T_7}(T^*_i - T_{7-i}T^*_7)h - T^*_7\tilde{F}^*_{7-i}D_{T^*_7}h\\
				&= D_{T_7}D_{T^*_7}\tilde{F}_iD_{T^*_7}h - T^*_7\tilde{F}^*_{7-i}D_{T^*_7}h\\
				&= (D_{T_7}D_{T^*_7}\tilde{F}_i - T^*_7\tilde{F}^*_{7-i})D_{T^*_7}h, 1\leq i \leq 6.
			\end{aligned}
		\end{equation}
		It yields from \eqref{2.3} that  $(F^*_iD_{T_7}D_{T^*_7} - F_{7-i}T^*_7)|_{\mathcal{D}_{T^*_7}} =
		D_{T_7}D_{T^*_7}\tilde{F}_i - T^*_7\tilde{F}^*_{7-i}$ for $1 \leqslant i \leqslant 6$. 

\end{enumerate}
This completes the proof.
	\end{proof}

We now prove the relationship between the fundamental operators of $\Gamma_{E(3; 3; 1, 1, 1)}$-contraction.
	\begin{thm}\label{Thm 1}
		Let  $F_i, 1\leq i \leq 6$ be fundamental operators of a $\Gamma_{E(3; 3; 1, 1, 1)}$-contraction $\textbf{T} = (T_1, \dots, T_7)$  and $\tilde{F}_j,1\leq j \leq 6$ be  fundamental operators of a $\Gamma_{E(3; 3; 1, 1, 1)}$-contraction $\textbf{T}^* = (T^*_1, \dots, T^*_7)$. If $[F_i, F_j] = 0$ for $1\leq i,j\leq 6$ and $\Ran T_7$ is dense in $\mathcal{H},$ then
		\begin{enumerate}
			\item $[\tilde{F}_i, \tilde{F}_j] = 0$ for $1\leq i,j\leq 6$,
			
			\item $[F_i, F^*_i] = [F_{7-i}, F^*_{7-i}]$ for $1 \leqslant i \leqslant 6$,
			
			\item $[\tilde{F}_i, \tilde{F}^*_i] = [\tilde{F}_{7-i}, \tilde{F}^*_{7-i}]$ for $1 \leqslant i \leqslant 6$.
		\end{enumerate}
	\end{thm}
	
	\begin{proof}
\begin{enumerate}

\item  As $\textbf{T} = (T_1, \dots, T_7)$ is a $\Gamma_{E(3; 3; 1, 1, 1)}$-contraction, it follows from Proposition \ref{Prop 3} that $T_7F_i = \tilde{F}^*_iT_7|_{\mathcal{D}_{T_7}}$ for $1 \leqslant i \leqslant 6$. Thus, we have
		\begin{equation}\label{2.4}
			\begin{aligned}
 \tilde{F}^*_j\tilde{F}^*_iT_7D_{T_7}&=T_7F_jF_iD_{T_7}\\&=T_7F_iF_jD_{T_7}
				\\&= \tilde{F}^*_iT_7F_jD_{T_7}(\text{since $F_i$ and $F_j$ commute for $1\leq i,j\leq 6$}) \\&= \tilde{F}^*_i\tilde{F}^*_jT_7D_{T_7}.
			\end{aligned}
		\end{equation}
It implies from \eqref{2.4} that for $1\leq i,j \leq 6$	
		\begin{equation}\label{2.6}
			\begin{aligned}
				&\tilde{F}^*_i\tilde{F}^*_jT_7D_{T_7}
				=
				\tilde{F}^*_j\tilde{F}^*_iT_7D_{T_7}\\
				 &\Rightarrow
				[\tilde{F}^*_i, \tilde{F}^*_j]D_{T^*_7}T_7 = 0\\
				&\Rightarrow[\tilde{F}_i, \tilde{F}_j]=0 (\text{since} \,\, \Ran T_7 \,\, \text{is dense in} \,\, \mathcal{H}).
			\end{aligned}
		\end{equation}
This completes the proof of part $(1)$ of the theorem.

\item By Proposition \ref{FiFj}, we observe that  $D_{T_7}T_i = F_iD_{T_7} + F^*_{7-i}D_{T_7}T_7$ for $1 \leqslant i \leqslant 6$. Multiplying $D_{T_7}F_{7-i}$ from left in both sides for $1\leq i \leq 6$,  we have
		\begin{equation}\label{2.7}
			\begin{aligned}
				&D_{T_7}F_{7-i}D_{T_7}T_i
				= D_{T_7}F_{7-i}F_iD_{T_7} + D_{T_7}F_{7-i}F^*_{7-i}D_{T_7}T_7\\
				&\Rightarrow (T_{7-i} - T^*_iT_7)T_i
				= D_{T_7}F_{7-i}F_iD_{T_7} + D_{T_7}F_{7-i}F^*_{7-i}D_{T_7}T_7\\
				&\Rightarrow T_{7-i}T_i - T^*_iT_iT_7 = D_{T_7}F_{7-i}F_iD_{T_7} + D_{T_7}F_{7-i}F^*_{7-i}D_{T_7}T_7.
			\end{aligned}
		\end{equation}
		Similarly, we also obtain
		\begin{equation}\label{2.8}
			\begin{aligned}
				&T_iT_{7-i} - T^*_{7-i}T_{7-i}T_7 = D_{T_7}F_iF_{7-i}D_{T_7} + D_{T_7}F_iF^*_iD_{T_7}T_7~{\rm{for}}~ 1\leq i \leq 6.
			\end{aligned}
		\end{equation}
Subtracting \eqref{2.8}-\eqref{2.7}, we get 	for $1\leq i \leq 6$
\begin{equation}\label{2.89}
\small{\begin{aligned}
(T_iT_{7-i}-T_{7-i}T_i )+(T^*_iT_i-T^*_{7-i}T_{7-i})T_7&= D_{T_7}[F_i,F_{7-i}]D_{T_7}+D_{T_7}(F_iF^*_i-F_{7-i}F^*_{7-i})D_{T_7}T_7
\end{aligned}}
\end{equation}
Since $T_iT_{7-i}=T_{7-i}T_i$  and $F_iF_{7-i}=F_{7-i}F_i$ for $1\leq i \leq 6,$  it follows from \eqref{2.89} that
		\begin{equation}\label{2.9}
			\begin{aligned}
				(T^*_iT_i - T^*_{7-i}T_{7-i})T_7
				&= D_{T_7}(F_iF^*_i - F_{7-i}F^*_{7-i})D_{T_7}T_7.
			\end{aligned}
		\end{equation}
It yields from Proposition \ref{F12} and \eqref{2.9} that for $1 \leqslant i \leqslant 6$ 
		\begin{equation}\label{2.10}
			\begin{aligned}
				&D_{T_7}(F^*_iF_i - F^*_{7-i}F_{7-i})D_{T_7}T_7 = D_{T_7}(F_iF^*_i - F_{7-i}F^*_{7-i})D_{T_7}T_7\\
				&\Rightarrow
				D_{T_7}([F_i, F^*_i] - [F_{7-i}, F^*_{7-i}])D_{T_7}T_7 = 0\\
				&\Rightarrow
				D_{T_7}([F_i, F^*_i] - [F_{7-i}, F^*_{7-i}])D_{T_7} = 0 \,\, (\text{since} \,\, \Ran T_7 \,\, \text{is dense in} \,\, \mathcal{H})\\
				&\Rightarrow
				[F_i, F^*_i] = [F_{7-i}, F^*_{7-i}].
			\end{aligned}
		\end{equation}
This completes the proof of part $(2)$ of the theorem.

\item By the Proposition \ref{Prop 3}, we have $D_{T_7}F_i = (T_iD_{T_7} - D_{T^*_7}\tilde{F}_{7-i}T_7)|_{\mathcal{D}_{T_7}}$. Multiplying $F_{7-i}D_{T_7}$ from the right in both sides,  we get
		\begin{equation}\label{2.11}
			\begin{aligned}
				D_{T_7}F_iF_{7-i}D_{T_7}
				&= T_iD_{T_7}F_{7-i}D_{T_7} - D_{T^*_7}\tilde{F}_{7-i}T_7F_{7-i}D_{T_7}\\
				&= T_i(T_{7-i} - T^*_iT_7) - D_{T^*_7}\tilde{F}_{7-i}\tilde{F}^*_{7-i}T_7D_{T_7}\\
				&= T_iT_{7-i} - T_iT^*_iT_7 - D_{T^*_7}\tilde{F}_{7-i}\tilde{F}^*_{7-i}D_{T^*_7}T_7~{\rm{for}}~1\leq i \leq 6.
			\end{aligned}
		\end{equation}
		Similarly, we also deduce that 
		\begin{equation}\label{2.13}
			\begin{aligned}
				D_{T_7}F_{7-i}F_iD_{T_7}
				= T_{7-i}T_i - T_{7-i}T^*_{7-i}T_7 - D_{T^*_7}\tilde{F}_i\tilde{F}^*_iD_{T^*_7}T_7~{\rm{for}}~1\leq i \leq 6.
			\end{aligned}
		\end{equation}
		By subtracting \eqref{2.11}-\eqref{2.13}, we obtain
		\begin{equation}\label{2.14}
			\begin{aligned}
				D_{T_7}[F_i, F_{7-i}]D_{T_7}
				&= D_{T^*_7}(\tilde{F}_i\tilde{F}^*_i - \tilde{F}_{7-i}\tilde{F}^*_{7-i})D_{T^*_7}T_7 - (T_iT^*_i - T_{7-i}T^*_{7-i})T_7 ~{\rm{for}}~1\leq i \leq 6.
			\end{aligned}
		\end{equation}
Since $(T^*_1, \dots, T^*_7)$ is a $\Gamma_{E(3; 3; 1, 1, 1)}$-contraction, it follows from Proposition \ref{F12} that 	
\begin{equation}\label{F21}
\begin{aligned}
(T_iT^*_i - T_{7-i}T^*_{7-i})&=D_{T^*_7}(\tilde{F}^*_i\tilde{F}_i - \tilde{F}^*_{7-i}\tilde{F}_{7-i})D_{T^*_7}~{\rm{for}}~1\leq i \leq 6.
\end{aligned}
\end{equation}
As $[F_i, F_{7-i}] = 0, 1\leq i \leq 6,$ we deduce from \eqref{2.14} and \eqref{F21} that 
		\begin{equation}\label{2.15}
			\begin{aligned}
				D_{T^*_7}(\tilde{F}_i\tilde{F}^*_i - \tilde{F}_{7-i}\tilde{F}^*_{7-i})D_{T^*_7}T_7 &= D_{T^*_7}(\tilde{F}^*_i\tilde{F}_i - \tilde{F}^*_{7-i}\tilde{F}_{7-i})D_{T^*_7}T_7 ~{\rm{for}}~1\leq i \leq 6
			\end{aligned}
		\end{equation}
which implies that 
		\begin{equation}\label{2.16}
			\begin{aligned}
				D_{T^*_7}([\tilde{F}_i, \tilde{F}^*_i] - [\tilde{F}_{7-i}, \tilde{F}^*_{7-i}])D_{T^*_7}T_7 &= 0.
			\end{aligned}
		\end{equation}
		Since $\Ran T_7$ is dense in $\mathcal{H}$, it follows that $[\tilde{F}_i, \tilde{F}^*_i] = [\tilde{F}_{7-i}, \tilde{F}^*_{7-i}]$ for $1 \leqslant i \leqslant 6$. This completes the proof of part $(3)$ of the theorem.
\end{enumerate}
Hence the proof of the theorem.		
		
	\end{proof}
	
We  present  a corollary to Theorem \ref{Thm 1} that  establishes a sufficient condition under which the  commutativity of the fundamental operators of a $\Gamma_{E(3; 3; 1, 1, 1)}$-contraction $\textbf{T} = (T_1, \dots, T_7)$ is  both necessary and sufficient for the commutativity of the fundamental operators of a $\Gamma_{E(3; 3; 1, 1, 1)}$-contraction $\textbf{T}^* = (T^*_1, \dots, T^*_7)$.
	
	\begin{cor}\label{Cor 1}
		Let $\textbf{T} = (T_1, \dots, T_7)$ be a $\Gamma_{E(3; 3; 1, 1, 1)}$-contraction on a Hilbert space $\mathcal{H}$ such that $T_7$ is invertible. Suppose that $F_i,1\leq i \leq 6$ are fundamental operators for $\textbf{T}$ and $\tilde{F}_j,1\leq j \leq 6$ are fundamental operators for $\textbf{T}^* = (T^*_1, \dots, T^*_7)$. Then $[F_i, F_j] = 0$ if and only if $[\tilde{F}_i, \tilde{F}_j] = 0$ for $1\leq i,j\leq 6.$
	\end{cor}
	
	\begin{proof}
We first assume that $[F_i, F_j] = 0$ for $1\leq i,j\leq 6.$ Since $T_7$ is invertible, it implies that $T_7$ has dense range. Furthermore, by Part $(1)$ of Theorem \ref{Thm 1}, we conclude that  $[\tilde{F}_i, \tilde{F}_j] = 0$ for $1\leq i,j\leq 6.$ 

Conversely, let $[\tilde{F}_i, \tilde{F}_j] = 0$ for $1\leq i,j\leq 6.$ As $T_7$ is invertible, it follows that $T_7^*$ possesses a dense range as well. By applying Theorem \ref{Thm 1} to the $\Gamma_{E(3; 3; 1, 1, 1)}$-contraction of  $\textbf{T}^* = (T^*_1, \dots, T^*_7)$, we conclude also $[F_i, F_j] = 0$ for $1\leq i,j\leq 6.$ This completes the proof.
	\end{proof}
		The following theorem establishes the relation between the fundamental operators of $\textbf{T}$ and $\textbf{T}^*$.
\begin{thm}\label{Thm 2}
		Let  $F_i, 1\leq i \leq 6$ be fundamental operators of a $\Gamma_{E(3; 3; 1, 1, 1)}$-contraction $\textbf{T} = (T_1, \dots, T_7)$  and $\tilde{F}_j,1\leq j \leq 6$ be  fundamental operators of a $\Gamma_{E(3; 3; 1, 1, 1)}$-contraction $\textbf{T}^* = (T^*_1, \dots, T^*_7)$. Then
		\begin{equation}\label{Fundamental P4}
			\begin{aligned}
				(F^*_i + F_{7-i}z)\Theta_{T^*_7}(z) &= \Theta_{T^*_7}(z)(\tilde{F}_i + \tilde{F}^*_{7-i}z) \,\, \textit{for} \,\, 1 \leqslant i \leqslant 6~{\rm{and~for ~all }}~ z \in \mathbb{D}.
			\end{aligned}
		\end{equation}
		
	\end{thm}
	
	\begin{proof}
		Note that
		\begin{equation}\label{2.17}
			\small{\begin{aligned}
				(F^*_i + F_{7-i}z)\Theta_{T^*_7}(z)
				&=
				(F^*_i + F_{7-i}z)(- T^*_7 + \sum_{n \geqslant 0} z^{n+1}D_{T_7}T^n_7D_{T^*_7})\\
				&= - F^*_iT^*_7 + z(- F_{7-i}T^*_7 + F^*_iD_{T_7}D_{T^*_7}) + \sum_{n \geqslant 2} z^n(F^*_iD_{T_7}T_7 + F_{7-i}D_{T_7})T^{n-2}_7D_{T^*_7}\\
				&= - T^*_7\tilde{F}_i + z(D_{T_7}D_{T^*_7}\tilde{F}_i - T^*_7\tilde{F}^*_{7-i}) + \sum_{n \geqslant 2} z^nD_{T_7}T_{7-i}T^{n-2}_7D_{T^*_7} (\text{applying Proposition \ref{Prop 3}})\\
				&= - T^*_7\tilde{F}_i + z(D_{T_7}D_{T^*_7}\tilde{F}_i - T^*_7\tilde{F}^*_{7-i}) \\&+ \sum_{n \geqslant 2} z^nD_{T_7}T^{n-2}_7(T_7D_{T^*_7}\tilde{F}_i + D_{T^*_7}\tilde{F}^*_{7-i})(\text{by Proposition \ref{FiFj}})\\
				&= \Theta_{T^*_7}(z)(\tilde{F}_i + \tilde{F}^*_{7-i}z), 1\leq i \leq 6.
			\end{aligned}}
		\end{equation}
		Therefore, $(F^*_i + F_{7-i}z)\Theta_{T^*_7}(z) = \Theta_{T^*_7}(z)(\tilde{F}_i + \tilde{F}^*_{7-i}z)$ for $1 \leqslant i \leqslant 6$ and $z \in \mathbb{D}$. This completes the proof.
	\end{proof}

We will now prove some important relations between fundamental operators of a $\Gamma_{E(3; 2; 1, 2)}$-contraction.
\begin{prop}[ Proposition $2.13$, \cite{apal2}]\label{Prop 2}
		Let $(S_1, S_2, S_3, \tilde{S_1}, \tilde{S_2})$ be a $\Gamma_{E(3; 2; 1, 2)}$-contraction. Then $(S_1, \tilde{S}_2, S_3), (\frac{\tilde{S}_1}{2}, \frac{S_2}{2}, S_3)$ and $(\frac{S_2}{2}, \frac{\tilde{S}_1}{2}, S_3)$ are $\Gamma_{E(2; 2; 1,1)}$-contractions.
	\end{prop}

\begin{prop}[ Lemma $2.9$, \cite{apal2}]\label{s1s3}
The fundamental operators  of a $\Gamma_{E(3; 2; 1, 2)}$-contraction $\textbf{S} = (S_1, S_2, S_3, \tilde{S}_1, \tilde{S}_2)$ are the unique operators $G_1,\tilde{G}_2,G_2$ and $\tilde{G}_1$  defined on $\mathcal{D}_{S_3}$ which satisfy the following operator equations
		\begin{equation}\label{s3}
			\begin{aligned}
				&D_{S_3}S_1 = G_1D_{S_3} + \tilde{G}_2^*D_{S_3}S_3, \,\, D_{S_3}\tilde{S}_2 = \tilde{G}_2D_{S_3} + G_1^*D_{S_3}S_3, \\&\,\, ~~~~~~~~~~~~~~~~~~~~~~~~~~~~~~~~~~~~~~~~~~~~\,\,\,\,\,\,\,\,\,\,\,\,\,\,\,\,\,\,\,\,\,\,\,\,\,\,\,\,\,\,\,\,\,\,\,\,\,\,\,\,\,\,\,\,\,\,\,\,\,\,\,\,\ \text{and}\\
				&D_{S_3}\frac{S_2}{2} = G_2D_{S_3} + \tilde{G}^*_1D_{S_3}S_3, \,\, D_{S_3}\frac{\tilde{S}_1}{2} = \tilde{G}_1D_{S_3} + G^*_2D_{S_3}S_3.
			\end{aligned}
		\end{equation}
		
	\end{prop}
	
	\begin{prop}[ Lemma $2.9$, \cite{apal2}] \label{ss11}
		Let $\textbf{S} = (S_1, S_2, S_3, \tilde{S}_1, \tilde{S}_2)$ be a $\Gamma_{E(3; 2; 1, 2)}$-contraction with commuting fundamental operators $G_1, \tilde{G}_2, G_2$ and $\tilde{G}_1$  defined on $\mathcal{D}_{S_3}$. Then
		\begin{equation}
			\begin{aligned}
				&S_1^*S_1 - \tilde{S}^*_2\tilde{S}_2 = D_{S_3}(G^*_1G_1 - \tilde{G}^*_2\tilde{G}_2)D_{S_3}, \\&\,\, ~~~~~~~~~~~~~~~~~~~~~~~~~~~~~~~~~~~~~~~~~~~~\,\,\,\,\,\,\,\,\,\,\,\,\,\,\,\,\,\,\,\,\,\,\,\,\,\,\,\,\,\,\,\,\,\,\,\,\,\,\,\,\,\,\,\,\,\,\,\,\,\,\,\ \text{and}\\
				&\frac{S^*_2S_2 - \tilde{S}^*_1\tilde{S}_1}{4} = D_{S_3}(G^*_2G_2 - \tilde{G}^*_1\tilde{G}_1)D_{S_3}.
			\end{aligned}
		\end{equation}
	\end{prop}
We now demonstrate the relationship among the fundamental operators of the $\Gamma_{E(3; 2; 1, 2)}$-contraction. The proof is similar to the Proposition \ref{Prop 3}. Therefore, we skip the proof.	
	\begin{prop}\label{Prop 7}
		Let  $G_1, 2G_2, 2\tilde{G}_1, \tilde{G}_2$ be the  fundamental operators for  a $\Gamma_{E(3; 2; 1, 2)}$-contraction $\textbf{S} = (S_1, S_2, S_3, \tilde{S}_1, \tilde{S}_2)$ defined on a Hilbert space $\mathcal H$  and $\hat{G}_1, 2\hat{G}_2, 2\hat{\tilde{G}}_1, \hat{\tilde{G}}_2$ be the fundamental operators for  a $\Gamma_{E(3; 2; 1, 2)}$-contraction $\textbf{S}^* = (S^*_1, S^*_2, S^*_3, \tilde{S}^*_1, \tilde{S}^*_2)$. Then the following properties hold:
		
		\begin{enumerate}
		\item $S_3G_1 = \hat{G}^*_1S_3|_{\mathcal{D}_{S_3}}, S_3G_2 = \hat{G}^*_2S_3|_{\mathcal{D}_{S_3}}, S_3\tilde{G}_1 = \hat{\tilde{G}}^*_1S_3|_{\mathcal{D}_{S_3}} ~{\rm{and}}~S_3\tilde{G}_2 = \hat{\tilde{G}}^*_2S_3|_{\mathcal{D}_{S_3}},$
			\item $(G^*_1D_{S_3}D_{S^*_3} - \tilde{G}_2S^*_3)|_{\mathcal{D}_{S^*_3}} = D_{S_3}D_{S^*_3}\hat{G}_1 - S^*_3\hat{\tilde{G}}^*_2$,
			
			\item $(G^*_2D_{S_3}D_{S^*_3} - \hat{\tilde{G}}_1S^*_3)|_{\mathcal{D}_{S^*_3}} = D_{S_3}D_{S^*_3}\hat{G}_2 - S^*_3\hat{\tilde{G}}^*_1$,
			
			\item $(\tilde{G}^*_1D_{S_3}D_{S^*_3} - G_2S^*_3)|_{\mathcal{D}_{S^*_3}} = D_{S_3}D_{S^*_3}\hat{\tilde{G}}_1 - S^*_3\hat{G}^*_2$,
			
			\item $(\tilde{G}^*_2D_{S_3}D_{S^*_3} - G_1S^*_3)|_{\mathcal{D}_{S^*_3}} = D_{S_3}D_{S^*_3}\hat{\tilde{G}}_2 - S^*_3\hat{G}^*_1$.
		\end{enumerate}
		
\end{prop}

We only state the following theorem. The proof is similar to Theorem \ref{Thm 1}. Therefore, we skip the prooof.	
		\begin{thm}\label{Thm 3}
		Let  $G_1, 2G_2, 2\tilde{G}_1, \tilde{G}_2$ be the  fundamental operators for  a $\Gamma_{E(3; 2; 1, 2)}$-contraction $\textbf{S} = (S_1, S_2, S_3, \tilde{S}_1, \tilde{S}_2)$ defined on a Hilbert space $\mathcal H$  and $\hat{G}_1, 2\hat{G}_2, 2\hat{\tilde{G}}_1, \hat{\tilde{G}}_2$ be the fundamental operators for  a $\Gamma_{E(3; 2; 1, 2)}$-contraction $\textbf{S}^* = (S^*_1, S^*_2, S^*_3, \tilde{S}^*_1, \tilde{S}^*_2)$. If $G_1, 2G_2, 2\tilde{G}_1, \tilde{G}_2$ commute with each other and $S_3$ has dense range, then
		\begin{enumerate}
			\item $\hat{G}_1, 2\hat{G}_2, 2\hat{\tilde{G}}_1, \hat{\tilde{G}}_2$ commute,
			
			\item $[G_1, G^*_1] = [\tilde{G}_2, \tilde{G}^*_2], [G_2, G^*_2] = [\tilde{G}_1, \tilde{G}^*_1]$,
			
			\item $[\hat{G}_1, \hat{G}^*_1] = [\hat{\tilde{G}}_2, \hat{\tilde{G}}^*_2], [\hat{G}_2, \hat{G}^*_2] = [\hat{\tilde{G}}_1, \hat{\tilde{G}}^*_1]$.
		\end{enumerate}
	\end{thm}
The following corollary provides a sufficient condition for the commutativity of the fundamental operators of a $\Gamma_{E(3; 2; 1, 2)} $-contraction  $\textbf{S} = (S_1, S_2, S_3, \tilde{S}_1, \tilde{S}_2) $ is both necessary and sufficient for the commutativity of the fundamental operators of a $\Gamma_{E(3; 2; 1, 2)} $-contraction $\textbf{S}^* = (S^*_1, S^*_2, S^*_3, \tilde{S}^*_1, \tilde{S}^*_2)$. The proof is same as the Corollary \ref{Cor 1}. Therefore, we skip the proof.
	\begin{cor}\label{Cor 2}
			Let  $G_1, 2G_2, 2\tilde{G}_1, \tilde{G}_2$ be the  fundamental operators for  a $\Gamma_{E(3; 2; 1, 2)}$-contraction $\textbf{S} = (S_1, S_2, S_3, \tilde{S}_1, \tilde{S}_2)$ defined on a Hilbert space $\mathcal H$  and $\hat{G}_1, 2\hat{G}_2, 2\hat{\tilde{G}}_1, \hat{\tilde{G}}_2$ be the fundamental operators for  a $\Gamma_{E(3; 2; 1, 2)}$-contraction $\textbf{S}^* = (S^*_1, S^*_2, S^*_3, \tilde{S}^*_1, \tilde{S}^*_2)$ with $S_3$ is invertible. Then $G_1, 2G_2, 2\tilde{G}_1, \tilde{G}_2$ commute with each other if and only if $\hat{G}_1, 2\hat{G}_2, 2\hat{\tilde{G}}_1, \hat{\tilde{G}}_2$ commute with each other.
	\end{cor}
	The following theorem establishes the relation between the fundamental operators of $\textbf{S}$ and $\textbf{S}^*$. The proof is same as the Theorem \ref{Thm 2}. Therefore, we skip the proof.
	\begin{thm}\label{Thm 4}
		Let $\textbf{S} = (S_1, S_2, S_3, \tilde{S}_1, \tilde{S}_2)$ be a $\Gamma_{E(3; 2; 1, 2)}$-contraction on a Hilbert space $\mathcal{H}$. Suppose $G_1, 2G_2, 2\tilde{G}_1, \tilde{G}_2$ and $\hat{G}_1, 2\hat{G}_2, 2\hat{\tilde{G}}_1, \hat{\tilde{G}}_2$ are fundamental operators for $\textbf{S}$ and $\textbf{S}^* = (S^*_1, S^*_2, S^*_3, \tilde{S}^*_1, \tilde{S}^*_2)$ respectively. Then for all $z \in \mathbb{D}$
		\begin{enumerate}
			\item $(G^*_1 + \tilde{G}_2z)\Theta_{S^*_3}(z) = \Theta_{S^*_3}(z)(\hat{G}_1 + \hat{\tilde{G}}^*_2z)$,
			
			\item $(G^*_2 + \tilde{G}_1z)\Theta_{S^*_3}(z) = \Theta_{S^*_3}(z)(\hat{G}_2 + \hat{\tilde{G}}^*_1z)$,
			
			\item $(\tilde{G}^*_1 + G_2z)\Theta_{S^*_3}(z) = \Theta_{S^*_3}(z)(\hat{\tilde{G}}_1 + \hat{G}^*_2z)$,
		
			\item $(\tilde{G}^*_2 + G_1z)\Theta_{S^*_3}(z) = \Theta_{S^*_3}(z)(\hat{\tilde{G}}_2 + \hat{G}^*_1z)$.
		\end{enumerate}
	\end{thm}

	\section{Functional Models for a pure $\Gamma_{E(3; 3; 1, 1, 1)}$-contraction and a pure $\Gamma_{E(3; 2; 1, 2)}$-contraction}\label{Section 3}
Sz.-Nagy and Foias \cite{Nagy} demonstrated that any pure contraction $T$ defined on a Hilbert space $\mathcal H$ is unitarily equivalent to the operator $\mathbb T=P_{\mathcal H_T}(M_z\otimes I)_{|_{\mathcal D_{T^*}}}$ on the Hilbert space $\mathcal H_{T}=(H^2(\mathbb D)\otimes \mathcal D_{T^*})\ominus M_{\Theta_{T}}(H^2(\mathbb D)\otimes \mathcal D_{T^*}),$ where $M_z$ denotes the multiplication operator on $H^2(\mathbb D)$ and $M_{\Theta_{T}}$ represents the multiplication operator from $H^2(\mathbb D)\otimes \mathcal D_{T}$ into $H^2(\mathbb D)\otimes \mathcal D_{T^*}$ associated with the multiplication $\Theta_{T}$, which is the characteristic function of $T,$ as defined in section $1.$		
In this section, we describe a model for a pure  $\Gamma_{E(3; 3; 1, 1, 1)}$-contraction and a pure $\Gamma_{E(3; 2; 1, 2)}$-contraction. 
	
We now produce functional model for a pure $\Gamma_{E(3; 3; 1, 1, 1)}$-contraction. In order to prove this, we  define $W : \mathcal{H} \to H^2(\mathbb{D}) \otimes \mathcal{D}_{T^*_7}$  by
	\begin{equation}\label{W}
		\begin{aligned}
			W(h) &= \sum_{n \geqslant 0} z^n \otimes D_{T^*_7}T^{*n}_7h.
		\end{aligned}
	\end{equation}
Since $T_7$ is a pure isometry, one can easily deduced that $W$ is isometry.  The adjoint of $W$ is given by 
	\begin{equation}\label{W*}
		\begin{aligned}
			W^*(z^n \otimes \xi) &= T^n_7D_{T^*_7}\xi \,\, \text{for} \,\, n \in \mathbb{N} \cup \{0\}, \xi \in \mathcal{D}_{T^*_7}.
		\end{aligned}
	\end{equation}

We only state the following lemma.  See \cite{SPal1} for the proof.
	\begin{lem}\label{Lem 1}
		Let $T_7$ be contraction. Then
		\begin{equation}\label{W Property}
			\begin{aligned}
				WW^* + M_{\Theta_{T_7}}M^*_{\Theta_{T_7}} = I_{H^2(\mathbb{D}) \otimes \mathcal{D}_{T^*_7}}.
			\end{aligned}
		\end{equation}
		
	\end{lem}
	
The following theorem describes the functional models for a pure $\Gamma_{E(3; 3; 1, 1, 1)}$-contraction. 
	
	\begin{thm}\label{Thm 5}
		Let $\textbf{T} = (T_1, \dots, T_7)$ be a $\Gamma_{E(3; 3; 1, 1, 1)}$-contraction on a Hilbert space $\mathcal{H}$. Suppose that $\tilde{F}_i, 1\leq i \leq 6$ are fundamental operators of $\textbf{T}^* = (T^*_1, \dots, T^*_7)$. Then
		\begin{enumerate}\label{Model 1}
			\item $T_i$  is unitarily equivalent to  $P_{\mathcal{H}_{T_7}}(I \otimes\tilde{ F}^*_i + M_z \otimes \tilde{F}_{7-i})_{|_{\mathcal{H}_{T_7}}}$ for $1 \leqslant i \leqslant 6$, and 
			
			\item $T_7$  is unitarily equivalent to $P_{\mathcal{H}_{T_7}}(M_z \otimes I_{\mathcal{D}_{T^*_7}})_{|_{\mathcal{H}_{T_7}}}$,
		\end{enumerate} where $\mathcal{H}_{T_7} =
			(H^2(\mathbb{D}) \otimes \mathcal{D}_{T^*_7}) \ominus M_{\Theta_{T_7}}(H^2(\mathbb{D}) \otimes \mathcal{D}_{T_7}).$
	\end{thm}
	
	\begin{proof}
		Since  $W$ is an isometry, it implies that  $WW^*$ is the projection onto the $\Ran W$. Also, as $T_7$ is a pure, it yields that $M_{\Theta_{T_7}}$ is an isometry. Thus, by Lemma \ref{Lem 1},  it follows that $W(\mathcal{H}) = \mathcal{H}_{T_7}$.  Note that 		\begin{equation}\label{M 1.1}
			\begin{aligned}
				W^*(I \otimes \tilde{ F}^*_i + M_z \otimes \tilde{F}_{7-i})(z^n \otimes \xi)
				&=
				W^*(z^n \otimes \tilde{F}^*_i\xi) + W^*(z^{n+1} \otimes \tilde{F}_{7-i}\xi)\\
				&= T^n_7D_{T^*_7}\tilde{F}^*_i\xi + T^{n+1}_7D_{T^*_7}\tilde{F}_{7-i}\xi\\
				&= T^n_7(D_{T^*_7}\tilde{F}^*_i + T_7D_{T^*_7}\tilde{F}_{7-i})\xi\\
				&= T^n_7(\tilde{F}_iD_{T^*_7} + \tilde{F}^*_{7-i}D_{T^*_7}T^*_7)^*\xi\\
				&= T^n_7(D_{T^*_7}T^*_i)^*\xi \,\, (\text{by Lemma 2.7 of \cite{ABD}})\\
				&= T_iT^n_7D_{T^*_7}\xi\\
				&= T_iW^*(z^n \otimes \xi)~{\rm{for}}~ 1 \leqslant i \leqslant 6.
			\end{aligned}
		\end{equation}
Thus, from \eqref{M 1.1}, we conclude that $W^*(I \otimes \tilde{ F}^*_i + M_z \otimes \tilde{F}_{7-i})=T_iW^*,1\leq i \leq 6$ on the vectors of the form $z^n \otimes \xi$ for all $n \geq 0$ and $\xi \in \mathcal{D}_{T^*_7},$ which span $H^2(\mathbb D)\otimes  \mathcal{D}_{T^*_7}$. This shows that $$W^*(I \otimes \tilde{ F}^*_i + M_z \otimes \tilde{F}_{7-i})=T_iW^*, 1\leq i \leq 6~{\rm  {on }}~~H^2(\mathbb D)\otimes  \mathcal{D}_{T^*_7}$$ and hence we have $W^*(I \otimes \tilde{ F}^*_i + M_z \otimes \tilde{F}_{7-i})W=T_i, 1\leq i \leq 6.$ Therfore, we deduce that $T_i$  is unitarily equivalent to  $P_{\mathcal{H}_{T_7}}(I \otimes\tilde{ F}^*_i + M_z \otimes \tilde{F}_{7-i})_{|_{\mathcal{H}_{T_7}}}$ for $1 \leqslant i \leqslant 6$. Observe that \begin{equation}\label{M 1.2}
			\begin{aligned}
				W^*(M_z \otimes I_{\mathcal{D}_{T^*_7}})(z^n \otimes \xi)
				&= W^*(z^{n+1} \otimes \xi)\\
				&= T^{n+1}_7D_{T^*_7}\xi\\
				&= T_7(T^n_7D_{T^*_7}\xi)\\
				&= T_7W^*(z^n \otimes \xi).
			\end{aligned}
		\end{equation}
Hence it follows from \eqref{M 1.2} that $W^*(M_z \otimes I_{\mathcal{D}_{T^*_7}})= T_7W^*$ on the vectors of the form $z^n \otimes \xi$ for all $n \geq 0$ and $\xi \in \mathcal{D}_{T^*_7}.$ By the same argument we also conclude that  $T_7$  is unitarily equivalent to $P_{\mathcal{H}_{T_7}}(M_z \otimes I_{\mathcal{D}_{T^*_7}})_{|_{\mathcal{H}_{T_7}}}$. This completes the proof.
	\end{proof}
	
	It is important to note that the unitary equivalence does not guarantee that the tuple
 \[\left(P_{\mathcal{H}_{T_7}}(I \otimes \tilde{ F}^*_1 + M_z \otimes \tilde{ F}_6)_{|_{\mathcal{H}_{T_7}}}, \dots ,P_{\mathcal{H}_{T_7}}(I \otimes \tilde{ F}^*_6 + M_z \otimes \tilde{ F}_1)_{|_{\mathcal{H}_{T_7}}}, P_{\mathcal{H}_{T_7}}(M_z \otimes I_{\mathcal{D}_{T^*_7}})_{|_{\mathcal{H}_{T_7}}}\right)\] constitues a commutative functional model. We observe that $P_{\mathcal{H}_{T_7}}(I \otimes \tilde{ F}^*_i + M_z \otimes \tilde{ F}_{7-i})|_{\mathcal{H}_{T_7}}$ commutes with $P_{\mathcal{H}_{T_7}}(M_z \otimes I_{\mathcal{D}_{T^*_7}})|_{\mathcal{H}_{T_7}}$ for all $1 \leqslant i \leqslant 6$. However,  $P_{\mathcal{H}_{T_7}}(I \otimes\tilde{ F}^*_i + M_z \otimes \tilde{F}_{7-i})|_{\mathcal{H}_{T_7}}$ commutes with $P_{\mathcal{H}_{T_7}}(I \otimes \tilde{F}^*_j + M_z \otimes \tilde{F}_{7-j})|_{\mathcal{H}_{T_7}}$  if and only if $[\tilde{F}_i, \tilde{F}_j] = 0$ and $[\tilde{F}^*_i, \tilde{F}_{7-j}] = [\tilde{F}^*_j, \tilde{F}_{7-i}]$ for $1 \leqslant i,j \leqslant 6$. 
	
	\begin{thm}\label{Thm 6}
		Let $\textbf{T} = (T_1, \dots, T_7)$ be a $\Gamma_{E(3; 3; 1, 1, 1)}$-contraction on a Hilbert space $\mathcal{H}$. Suppose that $\tilde{F}_i, 1\leq i \leq 6$ are fundamental operators of $\textbf{T}^* = (T^*_1, \dots, T^*_7)$ with $[\tilde{F}_i, \tilde{F}_j] = 0$ and $[\tilde{F}^*_i, \tilde{F}_{7-j}] = [\tilde{F}^*_j, \tilde{F}_{7-i}]$ for $1 \leqslant i,j \leqslant 6$. Then
		\begin{enumerate}\label{Model 1}
		  \item $\left(P_{\mathcal{H}_{T_7}}(I \otimes \tilde{ F}^*_1 + M_z \otimes \tilde{ F}_6)_{|_{\mathcal{H}_{T_7}}}, \dots ,P_{\mathcal{H}_{T_7}}(I \otimes \tilde{ F}^*_6 + M_z \otimes \tilde{ F}_1)_{|_{\mathcal{H}_{T_7}}}, P_{\mathcal{H}_{T_7}}(M_z \otimes I_{\mathcal{D}_{T^*_7}})_{|_{\mathcal{H}_{T_7}}}\right)$ is a $7$-tuple of commuting bounded operators,
		  
			\item $T_i$  is unitarily equivalent to  $P_{\mathcal{H}_{T_7}}(I \otimes\tilde{ F}^*_i + M_z \otimes \tilde{F}_{7-i})_{|_{\mathcal{H}_{T_7}}}$ for $1 \leqslant i \leqslant 6$, and 
			
			\item $T_7$  is unitarily equivalent to $P_{\mathcal{H}_{T_7}}(M_z \otimes I_{\mathcal{D}_{T^*_7}})_{|_{\mathcal{H}_{T_7}}}$.
		\end{enumerate}
		\end{thm}
The following corollary  provide an alternative proof of the Theorem $4.6$ \cite{apal2}.
	
	\begin{cor}\label{Cor 3}
		Let $\textbf{T} = (T_1, \dots, T_7)$ be a pure $\Gamma_{E(3; 3; 1, 1, 1)}$-isometry on a Hilbert space $\mathcal{H}$. Let $\tilde{F}_i, 1\leq i \leq 6$ be fundamental operators of $\textbf{T}^* = (T^*_1, \dots, T^*_7)$. Then $ (T_1, \dots, T_7)$ is unitarily equivalent to $(M_{\tilde{F}^*_1 + \tilde{F}_6z}, \dots, M_{\tilde{F}^*_6 + \tilde{F}_1z}, M_z)$. Furthermore, $\tilde{F}_1, \dots, \tilde{F}_6$ satisfy the following conditions:
\begin{enumerate}
\item $[\tilde{F}_i, \tilde{F}_j] = 0$  and 
\item  $[\tilde{F}^*_i, \tilde{F}_{7-j}] = [\tilde{F}^*_j, \tilde{F}_{7-i}]$ for $1 \leqslant i, j \leqslant 6$.
\end{enumerate}
	\end{cor}
	
	\begin{proof}
		Since $T_7$ is an isometry, the defect operator $D_{T_7} = 0$ and hence the defect space $\mathcal{D}_{T_7} = \{0\}$. As $T_7$ is an isometry, the characteristic function $\Theta_{T_7}$ equals zero. Thus, for an isometry $T_7$, the space $\mathcal{H}_{T_7}$ is equal to $H^2(\mathbb{D}) \otimes \mathcal{D}_{T^*_7}$. Therefore, it follows from Theorem \ref{Model 1} that $\textbf{T}$ is unitarily equivalent to $(M_{F^*_1 + F_6z}, \dots, M_{F^*_6 + F_1z}, M_z)$. As $(M_{F^*_1 + F_6z}, \dots, M_{F^*_6 + F_1z}, M_z)$ is commutative, it implies that $[F^*_i, F_{7-j}] = [F^*_j, F_{7-i}]$ for $1 \leqslant i, j \leqslant 6$. This completes the proof.
	\end{proof}
	
We now describe a functional model for pure $\Gamma_{E(3; 2; 1, 2)}$-contraction. To prove this, we define
 $\tilde{W} : \mathcal{H} \to H^2(\mathbb{D}) \otimes \mathcal{D}_{S^*_3}$ by \begin{equation}\label{W tilde}
		\begin{aligned}
			\tilde{W}(h) &= \sum_{n \geqslant 0} z^n \otimes D_{S^*_3}S^{*n}_3h
		\end{aligned}
	\end{equation}
As $S_3$ is an isometry, we deduce that $\tilde{W}$ is an isometry. The adjoint of $\tilde{W}^*$ has the following form
	\begin{equation}\label{W tilde*}
		\begin{aligned}
			\tilde{W}^*(z^n \otimes \eta) &= S^n_3D_{S^*_3}\eta \,\, \text{for} \,\, n \in \mathbb{N} \cup \{0\}, \eta \in \mathcal{D}_{S^*_3}.
		\end{aligned}
	\end{equation}
We also state the following lemma.  See \cite{SPal1} for the proof.	
	\begin{lem}\label{Lemm 1}
		Let $S_3$ be contraction. Then \begin{equation}\label{W tilde Property}
			\begin{aligned}
				\tilde{W}\tilde{W}^* + M_{\Theta_{S_3}}M^*_{\Theta_{S_3}} = I_{H^2(\mathbb{D}) \otimes \mathcal{D}_{S^*_3}}
			\end{aligned}
		\end{equation}
				
	\end{lem}	
Let $\hat{A}_1= P_{\mathcal{H}_{S_3}}(I \otimes \hat{G}^*_1 + M_z \otimes \hat{ \tilde{G}}_2)_{|_{\mathcal{H}_{S_3}}}, \hat{A}_2=P_{\mathcal{H}_{S_3}}(I \otimes 2\hat{G}^*_2 + M_z \otimes 2\hat{\tilde{G}}_1)_{|_{\mathcal{H}_{S_3}}}, \hat{A}_3=P_{\mathcal{H}_{S_3}}(M_z \otimes I_{\mathcal{D}_{S^*_3}})_{|_{\mathcal{H}_{S_3}}},\hat{B}_1=P_{\mathcal{H}_{S_3}}(I \otimes \hat{ \tilde{G}}^*_2 + M_z \otimes \hat{G}_1)_{|_{\mathcal{H}_{S_3}}},\hat{B}_2=P_{\mathcal{H}_{S_3}}(I \otimes \hat{ \tilde{G}}^*_2 + M_z \otimes \hat{G}_1)_{|_{\mathcal{H}_{S_3}}},$ where $\mathcal{H}_{S_3} =(H^2(\mathbb{D}) \otimes \mathcal{D}_{S^*_3}) \ominus M_{\Theta_{S_3}}(H^2(\mathbb{D}) \otimes \mathcal{D}_{S_3}).$ The following theorem demonstrates the functional models for a pure $\Gamma_{E(3; 3; 1, 1, 1)}$-contraction. The proof is similar to the proof of Theorem \ref{Model 1}. Therefore, we skip the proof.
\begin{thm}\label{Thm 7}
		Let $\textbf{S} = (S_1, S_2, S_3, \tilde{S}_1, \tilde{S}_2)$ be a $\Gamma_{E(3; 2; 1, 2)}$-contraction on a Hilbert space $\mathcal{H}$. Let $\hat{G}_1, 2\hat{G}_2, 2\hat{\tilde{G}}_1, \hat{\tilde{G}}_2$ be fundamental operators for $\textbf{S}^* = (S^*_1, S^*_2, S^*_3, \tilde{S}^*_1, \tilde{S}^*_2)$. Then 
		\begin{enumerate}\label{Model 2}
			\item $S_1$ is unitarily equivalent to $\hat{A}_1$,
			
			\item $S_2$  is unitarily equivalent to $\hat{A}_2$,
			
			\item $S_3$  is unitarily equivalent to $ \hat{A}_3$,
			
			\item $\tilde{S}_1$  is unitarily equivalent to $\hat{B}_1$,
			
			\item $\tilde{S}_2$  is unitarily equivalent to $\hat{B}_2$.
		\end{enumerate}

\end{thm}

 It is also interesting to notice that the unitary equivalence does not guarantee that the tuple 
$\left(\hat{A}_1,\hat{A}_2,\hat{A}_3,\hat{B}_1,\hat{B}_2\right)$ forms a commuting functional model.
However, the tuple $\left(\hat{A}_1,\hat{A}_2,\hat{A}_3,\hat{B}_1,\hat{B}_2\right)$ is  commutative if and only if $\hat{G}_1, 2\hat{G}_2, 2\hat{\tilde{G}}_1, \hat{\tilde{G}}_2$  commute with each other and $[\hat{G}_1, \hat{G}^*_1] = [\hat{\tilde{G}}_2, \hat{\tilde{G}}^*_2], [\hat{G}_2, \hat{G}^*_2] = [\hat{\tilde{G}}_1, \hat{\tilde{G}}^*_1], [\hat{G}_1, \hat{\tilde{G}}^*_1] = [\hat{G}_2, \hat{\tilde{G}}^*_2], [\hat{\tilde{G}}_1, \hat{G}^*_1] = [\hat{\tilde{G}}_2, \hat{G}^*_2], [\hat{G}_1, \hat{G}^*_2] = [\hat{\tilde{G}}_1, \hat{\tilde{G}}^*_2], [\hat{G}^*_1, \hat{G}_2] = [\hat{\tilde{G}}^*_1, \hat{\tilde{G}}_2]$. 
\begin{thm}\label{Thm 8}
		Let $\textbf{S} = (S_1, S_2, S_3, \tilde{S}_1, \tilde{S}_2)$ be a $\Gamma_{E(3; 2; 1, 2)}$-contraction on a Hilbert space $\mathcal{H}$. Suppose that $\hat{G}_1, 2\hat{G}_2, 2\hat{\tilde{G}}_1, \hat{\tilde{G}}_2$ are fundamental operators for $\textbf{S}^* = (S^*_1, S^*_2, S^*_3, \tilde{S}^*_1, \tilde{S}^*_2)$ with $\hat{G}_1, 2\hat{G}_2, 2\hat{\tilde{G}}_1, \hat{\tilde{G}}_2$  commute with each other and $[\hat{G}_1, \hat{G}^*_1] = [\hat{\tilde{G}}_2, \hat{\tilde{G}}^*_2], [\hat{G}_2, \hat{G}^*_2] = [\hat{\tilde{G}}_1, \hat{\tilde{G}}^*_1], [\hat{G}_1, \hat{\tilde{G}}^*_1] = [\hat{G}_2, \hat{\tilde{G}}^*_2], [\hat{\tilde{G}}_1, \hat{G}^*_1] = [\hat{\tilde{G}}_2, \hat{G}^*_2], [\hat{G}_1, \hat{G}^*_2] = [\hat{\tilde{G}}_1, \hat{\tilde{G}}^*_2], [\hat{G}^*_1, \hat{G}_2] = [\hat{\tilde{G}}^*_1, \hat{\tilde{G}}_2]$.  Then
		\begin{enumerate}\label{Commutative Model 2}
			\item  the tuple $\left(\hat{A}_1,\hat{A}_2,\hat{A}_3,\hat{B}_1,\hat{B}_2\right)$ is  commutative,
			\item $S_1$ is unitarily equivalent to $\hat{A}_1$,
			
			\item $S_2$  is unitarily equivalent to $\hat{A}_2$,
			
			\item $S_3$  is unitarily equivalent to $ \hat{A}_3$,
			
			\item $\tilde{S}_1$  is unitarily equivalent to $\hat{B}_1$,
			
			\item $\tilde{S}_2$  is unitarily equivalent to $\hat{B}_2$.
		\end{enumerate}

\end{thm}
The following corollary give an alternative proof of the Theorem $4.7$ \cite{apal2}. The proof is similar to the Corollary \ref{Cor 3}. Therefore, we skip the proof.

	\begin{cor}\label{Cor 4}
		Let $\textbf{S} = (S_1, S_2, S_3, \tilde{S}_1, \tilde{S}_2)$ be a pure $\Gamma_{E(3; 2; 1, 2)}$-isometry on a Hilbert space $\mathcal{H}$. Let $\hat{G}_1, 2\hat{G}_2, 2\hat{\tilde{G}}_1, \hat{\tilde{G}}_2$ be fundamental operators for $\textbf{S}^* = (S^*_1, S^*_2, S^*_3, \tilde{S}^*_1, \tilde{S}^*_2).$ Then $\textbf{S}$ is unitarily equivalent to $(M_{\hat{G}^*_1 + \hat{\tilde{G}}_2z}, M_{\hat{G}^*_2 + \hat{\tilde{G}}_1z}, M_z, M_{\hat{\tilde{G}}^*_1 + \hat{G}_2z}, M_{\hat{\tilde{G}}^*_2 + \hat{G}_1z})$. Furthermore, $\hat{G}_1, 2\hat{G}_2, 2\hat{\tilde{G}}_1, \hat{\tilde{G}}_2$ satisfy the following conditions: 
\begin{enumerate}
\item $\hat{G}_1, 2\hat{G}_2, 2\hat{\tilde{G}}_1, \hat{\tilde{G}}_2$  commute with each other, and

\item $[\hat{G}_1, \hat{G}^*_1] = [\hat{\tilde{G}}_2, \hat{\tilde{G}}^*_2], [\hat{G}_2, \hat{G}^*_2] = [\hat{\tilde{G}}_1, \hat{\tilde{G}}^*_1], [\hat{G}_1, \hat{\tilde{G}}^*_1] = [\hat{G}_2, \hat{\tilde{G}}^*_2], [\hat{\tilde{G}}_1, \hat{G}^*_1] = [\hat{\tilde{G}}_2, \hat{G}^*_2], [\hat{G}_1, \hat{G}^*_2] = [\hat{\tilde{G}}_1, \hat{\tilde{G}}^*_2], [\hat{G}^*_1, \hat{G}_2] = [\hat{\tilde{G}}^*_1, \hat{\tilde{G}}_2]$. 
\end{enumerate}		
\end{cor}
	
	\section{A Complete Set of Unitary Invariants}\label{Section 4}
Let $T$ and $T^{'}$ be contractions on Hilbert spaces $\mathcal{H}$ and $\mathcal{H}^{'}$, respectively. The characteristic functions of $T$ and $T^{'}$ are said to coincide if there exist unitary operators $U : \mathcal{D}_T \to \mathcal{D}_{T^{'}}$ and $U_* : \mathcal{D}_{T^*} \to \mathcal{D}_{T^{'*}}$ such that the following diagram commutes for all $z \in \mathbb{D}$ 
	\begin{equation}\label{Commutative Diagram}
		\begin{aligned}
			\begin{tikzcd}
				\mathcal{D}_T \arrow{r}{\Theta_T(z)} \arrow[swap]{d}{U} & \mathcal{D}_{T^*} \arrow{d}{U_*} \\
				\mathcal{D}_{T^{'}} \arrow{r}{\Theta_{T^{'}}(z)}& \mathcal{D}_{T^{'*}}.
			\end{tikzcd}
		\end{aligned}
	\end{equation}
The following result given by Sz.-Nagy and Foias \cite{Nagy} states that the characteristic function of a completely non-unitary contraction is a complete unitary invariant. 
\begin{thm}[Nagy-Foias]\label{NF Thm 12}
 Two completely non-unitary contractions are unitarily equivalent if and only if their characteristic functions coincide.
	\end{thm}	
In this section, we give a complete set of unitary invariant for a pure 	$\Gamma_{E(3; 3; 1, 1, 1)}$-contraction and a pure  $\Gamma_{E(3; 2; 1, 2)}$-contraction.

	\begin{prop}\label{Prop 13}
		If two $\Gamma_{E(3; 3; 1, 1, 1)}$-contractions $\textbf{T} = (T_1, \dots, T_7)$ and $\textbf{T}^{'} = (T^{'}_1, \dots, T^{'}_7)$ defined on  $\mathcal{H}$ and $\mathcal{H}^{'}$ respectively are unitarily equivalent, then  their fundamental operators $F_i,1\leq i \leq 6$ and $F^{'}_j,1\leq j \leq 6$ respectively are also unitarily equivalent.	\end{prop}
	
	\begin{proof}
Let  $U : \mathcal{H} \to \mathcal{H}^{'}$ be the unitary such that  $UT_i = T^{'}_iU$ for $1\leq i\leq 7.$ Then  we have $UT^*_i = T^{'*}_iU$ for $1 \leqslant i \leqslant 7$. We note that 
		\begin{equation}\label{UU}
			\begin{aligned}
				UD^2_{T_7} &= U(I - T^*_7T_7) = U - T^{'*}_7UT_7 = U - T^{'*}_7T^{'}_7U = D^2_{T^{'}_7}U.
			\end{aligned}
		\end{equation}
It follows from \eqref{UU} that  $UD_{T_7} = D_{T^{'}_7}U$. Let $\tilde{U} = U|_{\mathcal{D}_{T_7}}$. Then we have $\tilde{U}  \in \mathcal{B}(\mathcal{D}_{T_7}, \mathcal{D}_{T^{'}_7})$ and so $\tilde{U}D_{T_7} = D_{T^{'}_7}\tilde{U}$. Note that for $1\leq i \leq 6,$		\begin{equation}
			\begin{aligned}
				D_{T^{'}_7}\tilde{U}F_i\tilde{U}^*D_{T^{'}_7}
				&= \tilde{U}D_{T_7}F_iD_{T_7}\tilde{U}^*\\
				&= \tilde{U}(T_i - T^*_{7-i}T_7)\tilde{U}^* \\
				&= T^{'}_i - T^{'*}_{7-i}T^{'}_7\\
				&= D_{T^{'}_7}F^{'}_iD_{T^{'}_7}.
			\end{aligned}
		\end{equation}
Thus, we conclude that  $F^{'}_i = \tilde{U}F_i\tilde{U}^*$ for $1 \leqslant i \leqslant 6$.  This completes the proof.
	\end{proof}
The following proposition is a partial converse of the previous proposition for a pure $\Gamma_{E(3; 3; 1, 1, 1)}$-contraction.
\begin{prop}\label{Prop1 13}
		Let $\textbf{T} = (T_1, \dots, T_7)$ and $\textbf{T}^{'} = (T^{'}_1, \dots, T^{'}_7)$ be two pure $\Gamma_{E(3; 3; 1, 1, 1)}$-contractions on the Hilbert spaces $\mathcal{H}$ and $\mathcal{H}^{'},$ respectively, such that their characteristic functions of $T_7$ and $T^{'}_7$ coincide. Also, assume that the fundamental operators $(\tilde{F}_1,\dots, \tilde{F}_{6})$ of $\textbf{T}^* = (T^*_1, \dots, T^*_7)$ and $(F^{'}_{1*},\dots, F^{'}_{6*}) $ of $\textbf{T}^{'*} = (T^{'*}_1, \dots, T^{'*}_7)$ are unitarily equivalent by the same unitary that is involved in the coincidence of the characteristic functions of $T_7$ and $T^{'}_7.$ Then $\textbf{T}$ is unitarily equivalent to $\textbf{T}^{'}$.
\end{prop}
\begin{proof}
Let $U : \mathcal{D}_{T_7} \to \mathcal{D}_{T^{'}_7}$ and $U_* : \mathcal{D}_{T^*_7} \to \mathcal{D}_{T^{'*}_7}$ be unitary operators such that $U_*\tilde{F}_i = F^{'}_{i*}U_*$ for $1 \leqslant i \leqslant 6$ and $U_*\Theta_{T_7}(z) = \Theta_{T^{'}_7}(z)U$ for all $z \in \mathbb{D}$. Let $\tilde{U}_* :=I \otimes U^*: H^2(\mathbb{D}) \otimes \mathcal{D}_{T^*_7} \to H^2(\mathbb{D}) \otimes \mathcal{D}_{T^{'*}_7}$ be the operator defined by
		\begin{equation}\label{U tilde *}
			\begin{aligned}
				\tilde{U}_*(z^n \otimes \eta) = z^n \otimes U_*\eta \,\, \text{for} \,\, n \in \mathbb{N} \cup \{0\}, \eta \in \mathcal{D}_{T^*_7}.
			\end{aligned}
		\end{equation}
		Note that $\tilde{U}_*$ is a unitary and
		\begin{equation}\label{UI 1.1}
			\begin{aligned}
				\tilde{U}_*(M_{\Theta_{T_7}}f(z))
				&= \tilde{U}_*(\Theta_{T_7}(z)f(z))\\
				&= U_*\Theta_{T_7}(z)f(z)\\
				&= \Theta_{T^{'}_7}(z)Uf(z)\\
				&= M_{\Theta_{T^{'}_7}}(Uf(z))
			\end{aligned}
		\end{equation}
		for all $f \in H^2(\mathbb{D}) \otimes \mathcal{D}_{T_7}$ and $z \in \mathbb{D}$. It follows from \eqref{UI 1.1} that $\tilde{U}_*$ maps $\Ran M_{\Theta_{T_7}}$ onto $\Ran M_{\Theta_{T^{'}_7}}$. As $\tilde{U}_*$ is unitary, we conclude that
		\begin{equation}\label{UI 1.2}
			\begin{aligned}
				\tilde{U}_*(\mathcal{H}_{T_7})
				&= \tilde{U}_*((\Ran M_{\Theta_{T_7}})^{\perp})\\
				&= (\tilde{U}_*\Ran M_{\Theta_{T_7}})^{\perp}\\
				&= (\Ran M_{\Theta_{T^{'}_7}})^{\perp}\\
				&= \mathcal{H}_{T^{'}_7}.
			\end{aligned}
		\end{equation}
%
By definition of $\tilde{U}_*,$ we observe that for $1 \leqslant i \leqslant 6$,
		\begin{equation}\label{UI 1.3}
			\begin{aligned}
				\tilde{U}_*(I \otimes \tilde{F}^*_i + M_z \otimes \tilde{F}_{7-i})^*
				&= (I \otimes U_*)(I \otimes \tilde{F}_i + M^*_z \otimes \tilde{F}^*_{7-i})\\
				&= I \otimes U_*\tilde{F}_i + M^*_z \otimes U_*\tilde{F}^*_{7-i}\\
				&= I \otimes F^{'}_{i*}U_* + M^*_z \otimes F^{'*}_{(7-i)*}U_*\\
				&= (I \otimes F^{'*}_{i*} + M_z \otimes F^{'}_{(7-i )*})^*(I_{H^2} \otimes U_*)\\
				&= (I \otimes F^{'*}_{i*}+ M_z \otimes F^{'}_{(7-i)*})^*\tilde{U}_*.
			\end{aligned}
		\end{equation}
Also, by the definition of $\tilde{U}_*,$ it follows that
		\begin{equation}\label{UI 1.4}
			\begin{aligned}
				\tilde{U}_*(M_z \otimes I_{\mathcal{D}_{T^*_7}})
				&= (I \otimes U_*)(M_z \otimes I_{\mathcal{D}_{T^*_7}})\\
				&= (M_z \otimes I_{\mathcal{D}_{T^{'*}_7}})(I_{H^2} \otimes U_*)\\
				&= (M_z \otimes I_{\mathcal{D}_{T^{'*}_7}})\tilde{U}_*.
			\end{aligned}
		\end{equation}
Thus, $\mathcal{H}_{T^{'}_7} = \tilde{U}_*(\mathcal{H}_{T_7})$ is a co-invariant subspace of
		\[(I \otimes F^{'*}_i + M_z \otimes F^{'}_{7-i}) \,\,\text{for}\,\, 1 \leqslant i \leqslant 6 \,\,\text{and}\,\, (M_z \otimes I_{\mathcal{D}_{T_7}}).\]
Consequently, we derive
		\begin{equation*}
			\begin{aligned}
				P_{\mathcal{H}_{T_7}}(I \otimes F^*_i + M_z \otimes F_{7-i})_{|_{\mathcal{H}_{T_7}}}
				&\cong P_{\mathcal{H}_{T^{'}_7}}(I \otimes F^{'*}_i + M_z \otimes F^{'}_{7-i})_{|_{\mathcal{H}_{T^{'}_7}}}
			\end{aligned}
		\end{equation*}
		for $1 \leqslant i \leqslant 6$ and
		\begin{equation*}
			\begin{aligned}
				P_{\mathcal{H}_{T_7}}(M_z \otimes I_{\mathcal{D}_{T^*_7}})_{|_{\mathcal{H}_{T_7}}}
				&\cong P_{\mathcal{H}_{T^{'}_7}}(M_z \otimes I_{\mathcal{D}_{T^{'*}_7}})_{|_{\mathcal{H}_{T^{'}_7}}},
			\end{aligned}
		\end{equation*}
and the corresponding unitary operator that unitarizes them is $U_* : \mathcal{D}_{T^*_7} \to \mathcal{D}_{T^{'*}_7}$. Therefore, $\textbf{T}$ and $\textbf{T}^{'}$ are unitarily equivalent. This completes the proof.

\end{proof}
Combining the Proposition \ref{Prop 13} and Proposition \ref{Prop1 13}, we prove the main result of this section, the unitary invariance for a pure $\Gamma_{E(3; 3; 1, 1, 1)}$-contraction.
\begin{thm}\label{Thm 123}
Let $\textbf{T} = (T_1, \dots, T_7)$ and $\textbf{T}^{'} = (T^{'}_1, \dots, T^{'}_7)$ be two pure $\Gamma_{E(3; 3; 1, 1, 1)}$-contractions on the Hilbert spaces $\mathcal{H}$ and $\mathcal{H}^{'},$ respectively. Suppose  $(\tilde{F}_1,\dots, \tilde{F}_{6})$ and $(F^{'}_{1*},\dots, F^{'}_{6*}) $ are fundamental operators of $\textbf{T}^* = (T^*_1, \dots, T^*_7)$ and $\textbf{T}^{'*} = (T^{'*}_1, \dots, T^{'*}_7),$ respectively. Then $\textbf{T}$ is unitarily equivalent to $\textbf{T}^{'}$ if and only if the characteristic functions of $T_7$ and $T^{'}_7$ coincide and   $(\tilde{F}_1,\dots, \tilde{F}_{6})$ 
is unitarily equivalent to  $(F^{'}_{1*},\dots, F^{'}_{6*}) $ by the same unitary that is involved in the coincidence of the characteristic functions of $T_7$ and $T^{'}_7.$
\end{thm}		
\begin{proof}
Since $\textbf{T}$ is unitarily equivalent to $\textbf{T}^{'}$, so are $\textbf{T}^* = (T^*_1, \dots, T^*_7)$ and $\textbf{T}^{'*} = (T^{'*}_1, \dots, T^{'*}_7).$ It follows from Proposition \ref{Prop 13}  that  $(\tilde{F}_1,\dots, \tilde{F}_{6})$ and $(F^{'}_{1*},\dots, F^{'}_{6*}) $ are unitarily equivalence. This completes the proof.

\end{proof}

We now discuss a complete set of unitary invariant for  a pure  $\Gamma_{E(3; 2; 1, 2)}$-contraction. The proof of the following proposition is similar to the Proposition \ref{Prop 13}. Therefore,  we skip the proof.
	
	\begin{prop}\label{Prop 14}
		If two $\Gamma_{E(3; 2; 1, 2)} $-contraction $\textbf{S} = (S_1, S_2, S_3, \tilde{S}_1, \tilde{S}_2) $ and $\textbf{S}^{'} = (S^{'}_1, S^{'}_2, S^{'}_3, \tilde{S}^{'}_1, \tilde{S}^{'}_2)$ acting on the Hilbert spaces $\mathcal{H}$ and $\mathcal{H}^{'},$ respectively, are unitarily equivalent, then so are their fundamental operators $(G_1, 2G_2, 2\tilde{G}_1, \tilde{G}_2)$ and $(G^{'}_1, 2G^{'}_2, 2\tilde{G}^{'}_1, \tilde{G}^{'}_2),$ respectively.	\end{prop}
The following proposition is a partial converse of the previous proposition for a pure $\Gamma_{E(3; 2; 1, 2)}$-contraction. The proof of the following proposition is same as Proposition \ref{Prop1 13}. Therefore, we skip the proof.
	
\begin{prop}\label{Prop1 15}
		Let $\textbf{S} = (S_1, S_2, S_3, \tilde{S}_1, \tilde{S}_2)$ and $\textbf{S} = (S^{'}_1, S^{'}_2, S^{'}_3, \tilde{S}^{'}_1, \tilde{S}^{'}_2)$ be two pure $\Gamma_{E(3; 2; 1, 2)}$-contractions on the Hilbert spaces $\mathcal{H}$ and $\mathcal{H}^{'},$ respectively, such that their  characteristic functions of $S_3$ and $ S^{'}_3$ coincide. Also, suppose that the fundamental operators $(\hat{G}_1, 2\hat{G}_2, 2\hat{\tilde{G}}_1, \hat{\tilde{G}}_2)$ of $\textbf{S}^* = (S^*_1, S^*_2, S^*_3, \tilde{S}^*_1, \tilde{S}^*_2)$ and $(G^{'}_{1*}, 2G^{'}_{2*}, 2\tilde{G}^{'}_{1*}, \tilde{G}^{'}_{2*})$ of $\textbf{S}^{'*} = (S^{'*}_1, S^{'*}_2, S^{'*}_3, \tilde{S}^{'*}_1, \tilde{S}^{'*}_2)$ are unitarily equivalent by the same unitary that is involved in the coincidence of the characteristic functions of $S_3$ and $S^{'}_3.$ Then $\textbf{S}$ is unitarily equivalent to $\textbf{S}^{'}$.

\end{prop}

Combining the Proposition \ref{Prop 14} and Proposition \ref{Prop1 15}, we demonstrate the main result of this section, the unitary invariance for a pure $\Gamma_{E(3; 2; 1, 2)}$-contraction.	The proof is similar to the Theorem \ref{Thm 123}. Therefore, we skip the proof.
	\begin{thm}\label{Thm 14}
		Let $\textbf{S} = (S_1, S_2, S_3, \tilde{S}_1, \tilde{S}_2)$ and $\textbf{S} = (S^{'}_1, S^{'}_2, S^{'}_3, \tilde{S}^{'}_1, \tilde{S}^{'}_2)$ be two pure $\Gamma_{E(3; 2; 1, 2)}$-contractions on the Hilbert spaces $\mathcal{H}$ and $\mathcal{H}^{'},$ respectively. Assume that the fundamental operators $(\hat{G}_1, 2\hat{G}_2, 2\hat{\tilde{G}}_1, \hat{\tilde{G}}_2)$ and $(G^{'}_{1*}, 2G^{'}_{2*}, 2\tilde{G}^{'}_{1*}, \tilde{G}^{'}_{2*})$ of  $\textbf{S}^* = (S^*_1, S^*_2, S^*_3, \tilde{S}^*_1, \tilde{S}^*_2)$ and  $\textbf{S}^{'*} = (S^{'*}_1, S^{'*}_2, S^{'*}_3, \tilde{S}^{'*}_1, \tilde{S}^{'*}_2),$ respectively. Then $\textbf{S}$ is unitarily equivalent to $\textbf{S}^{'}$ if and only if the characteristic functions of $S_3$ and $S_3^{\prime}$ coincide and $(\hat{G}_1, 2\hat{G}_2, 2\hat{\tilde{G}}_1, \hat{\tilde{G}}_2)$ of $\textbf{S}^* = (S^*_1, S^*_2, S^*_3, \tilde{S}^*_1, \tilde{S}^*_2)$ and $(G^{'}_{1*}, 2G^{'}_{2*}, 2\tilde{G}^{'}_{1*}, \tilde{G}^{'}_{2*})$ of $\textbf{S}^{'*} = (S^{'*}_1, S^{'*}_2, S^{'*}_3, \tilde{S}^{'*}_1, \tilde{S}^{'*}_2)$ are unitarily equivalent by the same unitary that is involved in the coincidence of the characteristic functions of $S_3$ and $S^{'}_3.$	\end{thm}

\section{Abstract Models for Special Classes of c.n.u. $\Gamma_{E(3; 3; 1, 1, 1)}$-contraction, c.n.u. $\Gamma_{E(3; 2; 1, 2)}$-contraction and c.n.u. Tetrablock contraction}
	
	In this section, we construct of an operator model for a certain class of c.n.u. $\Gamma_{E(3; 3; 1, 1, 1)}$-contraction, c.n.u $\Gamma_{E(3; 2; 1, 2)}$-contraction and c.n.u. tetrablock contraction. 	A model for a class of c.n.u. $\Gamma_n$-contraction $(S_1, \dots, S_{n-1}, P)$ that satisfying
	\begin{equation}\label{Condition 4}
		\begin{aligned}
			S^*_iP = PS^*_i \,\, \text{for} \,\, 1 \leqslant i \leqslant n-1
		\end{aligned}
	\end{equation}
	can be found in [Theorem 4.5, \cite{bisai2}]. Let $\mathcal{A}, \mathcal{A}_*$ be defined as
	\begin{equation}\label{A, A_*}
		\begin{aligned}
			\mathcal{A} &= SOT-\lim_{n \to \infty} T^{*n}_7T^n_7 \,\,\text{and}\,\, \mathcal{A}_* = SOT-\lim_{n \to \infty} T^n_7T^{*n}_7.
		\end{aligned}
	\end{equation}
Define an operator $V: \overline{\Ran} \mathcal{A} \rightarrow \overline{\Ran} \mathcal{A}$ by
	\begin{equation}\label{V}
		\begin{aligned}
			V(\mathcal{A}^{1/2}x)
			&= \mathcal{A}^{1/2}T_7x.
		\end{aligned}
	\end{equation}
Observe that \begin{equation}\label{Property A, A_*,  A tilde, A tilde_*}
		\begin{aligned}
			&\mathcal{A}^{1/2}\mathcal{A}_*\mathcal{A}^{1/2}V(\mathcal{A}^{1/2}x)
			= \mathcal{A}^{1/2}\mathcal{A}_*\mathcal{A}T_7x.
			\end{aligned}\end{equation}
	We define $Q : \overline{\Ran} \mathcal{A} \to \overline{\Ran} \mathcal{A}$ by
	\begin{equation}\label{Q}
		\begin{aligned}
			Qx
			&= (I - \mathcal{A}^{1/2}\mathcal{A}_*\mathcal{A}^{1/2})^{1/2}x.
		\end{aligned}
	\end{equation}	
We only state the following proposition. The proof is similar to Theorem \ref{Thm 5}. Therefore, we skip the proof.
\begin{prop}\label{Prop 9}
		Let $\textbf{T} = (T_1, \dots, T_7)$ be a $\Gamma_{E(3; 3; 1, 1, 1)}$-contraction on a Hilbert space $\mathcal{H}$. Let $F_1, \dots, F_6$ and $\tilde{F}_1, \dots, \tilde{F}_6$ be the fundamental operators of $\textbf{T}$ and $\textbf{T}^*$ respectively. Then
		\begin{enumerate}\label{Property 1.1}
			\item $\tilde{F}^*_iD_{T^*_7}\mathcal{A}^{1/2}|_{\overline{\Ran}\mathcal{A}}+T_7T^*_7\tilde{F}_{7-i}D_{T^*_7}\mathcal{A}^{1/2}V = D_{T^*_7}T_i\mathcal{A}^{1/2}|_{\overline{\Ran}\mathcal{A}}$,
			
			\item $\tilde{F^*_i}D_{T^*_7}T^*_7 + T_7T^*_7\tilde{F}_{7-i}D_{T^*_7}
			= D_{T^*_7}T_iT^*_7$
		\end{enumerate}
		for $1 \leqslant i \leqslant 6$.
	\end{prop}
We recall the following theorem from \cite{Durszt}.
\begin{thm}[Durszt, \cite{Durszt}]\label{Thm 9}
		If $T$ is a c.n.u. contraction on Hilbert space $\mathcal{H}$ then there exists an isometry $W : \mathcal{H} \to (H^2(\mathbb{D}) \otimes \mathcal{D}_T) \oplus (L^2(\mathbb{T}) \otimes \mathcal{D}_{T^*})$ such that
		\begin{equation}\label{Model CNU}
			\begin{aligned}
				WT &=
				((M^*_z \otimes I_{\mathcal{D}_T}) \oplus (M^*_{e^{it}} \otimes I_{\mathcal{D}_{T^*}}))W.
			\end{aligned}
		\end{equation}
	\end{thm}
	It is important to observe that $W$ has two components. Let $W = (W_1, W_2)$, where
	$W_1: \mathcal{H} \to H^2(\mathbb{D}) \otimes \mathcal{D}_{T_7}$ and $W_2 : \mathcal{H}_0 \to L^2(\mathbb{T}) \otimes \mathcal{D}_{T^*_7}$ are given by 
	\begin{equation}\label{(W_1, W_2)}
		\begin{aligned}
			&W_1h
			= \sum_{n \geqslant 0} z^n \otimes D_{T_7}T^n_7h, {\rm{and}}\\
			&W_2x
			= \sum_{n \leqslant -1} z^n \otimes D_{T^*_7}\mathcal{A}^{1/2}Q^{-1}V^{*-n}\mathcal{A}^{1/2}x + \sum_{n \geqslant 0} z^n \otimes D_{T^*_7}\mathcal{A}^{1/2}Q^{-1}V^n\mathcal{A}^{1/2}x.
		\end{aligned}
	\end{equation}	
We also describe a model for completely nonunitary $\Gamma_{E(3; 3; 1, 1, 1)}$-contraction. The proof is similar to Theorem \ref{Thm 5}. We therefore skip the proof.
\begin{thm}[Model for Special c.n.u. $\Gamma_{E(3; 3; 1, 1, 1)}$-Contraction]\label{Thm 10}
		Let $\textbf{T} = (T_1, \dots, T_7)$ be a c.n.u. $\Gamma_{E(3; 3; 1, 1, 1)}$-contraction on a Hilbert space $\mathcal{H}$ with $T^*_iT_7 = T_7T^*_i$ for $1 \leqslant i \leqslant 6$. Let $F_1, \dots, F_6$ and $\tilde{F}_1, \dots, \tilde{F}_6$ be the fundamental operators of $\textbf{T}$ and $\textbf{T}^*$ respectively. Consider $W = (W_1, W_2)$ as above and let $\mathcal{L} = \Ran W$. Then
		\begin{enumerate}\label{Model 3}
			\item $T_i \cong ((I \otimes F_i + M^*_z\otimes F^*_{7-i})\oplus (I \otimes \tilde{F}^*_i + M^*_{e^{it}}\otimes T_7T^*_7\tilde{F}_{7-i}))|_{\mathcal{L}}$ for $1 \leqslant i \leqslant 6$,
			
			\item $T_7 \cong ((M^*_z \otimes I_{\mathcal{D}_{T_7}}) \oplus (M^*_{e^{it}} \otimes I_{\mathcal{D}_{T^*_7}}))|_{\mathcal{L}}$.
		\end{enumerate}
	\end{thm}
The following theorem gives the unitary invariance of a completely nonunitary $\Gamma_{E(3; 3; 1, 1, 1)}$-contraction. The proof is similar to Theorem \ref{Thm 123}. Therefore, we omit the proof.

\begin{thm}\label{Thm 15}
		Let $\textbf{T} = (T_1, \dots, T_7)$ and $\textbf{T}^{'} = (T^{'}_1, \dots, T^{'}_7)$ be two $\Gamma_{E(3; 3; 1, 1, 1)}$-contractions on the Hilbert spaces $\mathcal{H}$ and $\mathcal{H}^{'}$ respectively. Suppose $F_1, \dots, F_6$ and $F^{'}_1, \dots, F^{'}_6$ are the fundamental operators for $\textbf{T}$ and $\textbf{T}^{'}$ respectively; and $\tilde{F}_1, \dots, \tilde{F}_6$ and $\tilde{F}^{'}_1, \dots, \tilde{F}^{'}_6$ are the fundamental operators for $\textbf{T}^*$ and $\textbf{T}^{'*}$ respectively. Then $\textbf{T}$ and $\textbf{T}^{'}$ are unitarily equivalent if and only if the characteristic tuples of $\textbf{T}$ and $\textbf{T}^{'}$ are unitarily equivalent and the fundamental operators $\tilde{F}_1, \dots, \tilde{F}_6$ are unitarily equivalent to $\tilde{F}^{'}_1, \dots, \tilde{F}^{'}_6$ respectively.
	\end{thm}
		
Let $\tilde{\mathcal{A}}, \tilde{\mathcal{A}}_*$ be defined as follows:
	\begin{equation}\label{A tilde, A tilde_*}
		\begin{aligned}
			\tilde{\mathcal{A}} &= SOT-\lim_{n \to \infty} S^{*n}_3S^n_3 \,\,\text{and}\,\, \tilde{\mathcal{A}}_* = SOT-\lim_{n \to \infty} S^n_3S^{*n}_3.
		\end{aligned}
	\end{equation}
	Define an operator  $\tilde{V}: \overline{\Ran} \tilde{\mathcal{A}} \rightarrow \overline{\Ran} \tilde{\mathcal{A}}$ by
	\begin{equation}\label{V tilde}
		\begin{aligned}
			\tilde{V}(\tilde{\mathcal{A}}^{1/2}x)
			&= \tilde{\mathcal{A}}^{1/2}S_3x.
		\end{aligned}
	\end{equation}
We note that 
	\begin{equation}\label{Property A, A_*,  A tilde, A tilde_*}
		\begin{aligned}
		&\tilde{\mathcal{A}}^{1/2}\tilde{\mathcal{A}}_*\tilde{\mathcal{A}}^{1/2}\tilde{V}(\tilde{\mathcal{A}}^{1/2}x)
		= \tilde{\mathcal{A}}^{1/2}\tilde{\mathcal{A}}_*\tilde{\mathcal{A}}S_3x.
		\end{aligned}
	\end{equation}
	We define  $\tilde{Q} : \overline{\Ran} \tilde{\mathcal{A}} \to \overline{\Ran} \tilde{\mathcal{A}}$ by
	\begin{equation}\label{Q tilde}
		\begin{aligned}
			\tilde{Q}x
			&= (I - \tilde{\mathcal{A}}^{1/2}\tilde{\mathcal{A}}_*\tilde{\mathcal{A}}^{1/2})^{1/2}x.
		\end{aligned}
	\end{equation}
We only state the following proposition. The proof is similar to Theorem \ref{Thm 7}. Therefore, we skip the proof.
\begin{prop}\label{Prop 10}
		Let $\textbf{S} = (S_1, S_2, S_3, \tilde{S}_1, \tilde{S}_2)$ be a $\Gamma_{E(3; 2; 1, 2)}$-contraction on a Hilbert space $\mathcal{H}$. Let $G_1, 2G_2, 2\tilde{G}_1, \tilde{G}_2$ and $\hat{G}_1, 2\hat{G}_2, 2\hat{\tilde{G}}_1, \hat{\tilde{G}}_2$ be the fundamental operators of $\textbf{S}$ and $\textbf{S}^*$ respectively. Then we have the following:
		\begin{enumerate}
			\item $\hat{G}^*_1D_{S^*_3}\tilde{\mathcal{A}}^{1/2}|_{\overline{\Ran}\tilde{\mathcal{A}}} + S_3S^*_3\hat{\tilde{G}}_2D_{S^*_3}\tilde{\mathcal{A}}^{1/2}\tilde{V} = D_{S^*_3}S_1\tilde{\mathcal{A}}^{1/2}|_{\overline{\Ran}\tilde{\mathcal{A}}}$,
			
			\item $2\hat{G}^*_2D_{S^*_3}\tilde{\mathcal{A}}^{1/2}|_{\overline{\Ran}\tilde{\mathcal{A}}} + 2S_3S^*_3\hat{\tilde{G}}_1D_{S^*_3}\tilde{\mathcal{A}}^{1/2}\tilde{V} = D_{S^*_3}S_2\tilde{\mathcal{A}}^{1/2}|_{\overline{\Ran}\tilde{\mathcal{A}}}$,
			
			\item $2\hat{\tilde{G}}^*_1D_{S^*_3}\tilde{\mathcal{A}}^{1/2}|_{\overline{\Ran}\tilde{\mathcal{A}}} + 2S_3S^*_3\hat{G}_2D_{S^*_3}\tilde{\mathcal{A}}^{1/2}\tilde{V} = D_{S^*_3}\tilde{S}_1\tilde{\mathcal{A}}^{1/2}|_{\overline{\Ran}\tilde{\mathcal{A}}}$,
			
			\item $\hat{\tilde{G}}^*_2D_{S^*_3}\tilde{\mathcal{A}}^{1/2}|_{\overline{\Ran}\tilde{\mathcal{A}}} + S_3S^*_3\hat{G}_1D_{S^*_3}\tilde{\mathcal{A}}^{1/2}\tilde{V} = D_{S^*_3}\tilde{S}_2\tilde{\mathcal{A}}^{1/2}|_{\overline{\Ran}\tilde{\mathcal{A}}}$,
			
			\item $\hat{G}^*_1D_{S^*_3}S^*_3 + S_3S^*_3\hat{\tilde{G}}_2D_{S^*_3}
			= D_{S^*_3}S_1S^*_3$,
			
			\item $2\hat{G}^*_2D_{S^*_3}S^*_3 + 2S_3S^*_3\hat{\tilde{G}}_1D_{S^*_3}
			= D_{S^*_3}S_2S^*_3$,
			
			\item $2\hat{\tilde{G}}^*_1D_{S^*_3}S^*_3 + 2S_3S^*_3\hat{G}_2D_{S^*_3}
			= D_{S^*_3}\tilde{S}_1S^*_3$,
			
			\item $\hat{\tilde{G}}^*_2D_{S^*_3}S^*_3 + S_3S^*_3\hat{G}_1D_{S^*_3}
			= D_{S^*_3}\tilde{S}_2S^*_3$.
		\end{enumerate}
	\end{prop}
		
It is important to observe that $\tilde{W}$ has two components.  Let $\tilde{W} = (\tilde{W}_1, \tilde{W}_2)$, where 
 $\tilde{W}_1 : \mathcal{H} \to H^2(\mathbb{D}) \otimes \mathcal{D}_{S_3}$ and $\tilde{W}_2 : \tilde{\mathcal{H}}_0 \to L^2(\mathbb{T}) \otimes \mathcal{D}_{S^*_3}$ are given by
	\begin{equation}\label{(W_1 tilde, W_2 tilde)}
		\begin{aligned}
			&\tilde{W}_1h
			= \sum_{n \geqslant 0} z^n \otimes D_{S_3}S^n_3h,\\
			&\tilde{W}_2x
			= \sum_{n \leqslant -1} z^n \otimes D_{S^*_3}\tilde{\mathcal{A}}^{1/2}\tilde{Q}^{-1}\tilde{V}^{*-n}\tilde{\mathcal{A}}^{1/2}x + \sum_{n \geqslant 0} z^n \otimes D_{S^*_3}\tilde{\mathcal{A}}^{1/2}\tilde{Q}^{-1}\tilde{V}^n\tilde{\mathcal{A}}^{1/2}x.
		\end{aligned}
	\end{equation}
We only state the following theorem. The proof is similar to Theorem \ref{Thm 7}. Therefore, we skip the proof.
	\begin{thm}[Model for special c.n.u $\Gamma_{E(3; 2; 1, 2)}$-contraction]\label{Thm 11}
		Let $\textbf{S} = (S_1, S_2, S_3, \tilde{S}_1, \tilde{S}_2)$ be a c.n.u. $\Gamma_{E(3; 2; 1, 2)}$-contraction on a Hilbert space $\mathcal{H}$ with $S^*_iS_3 = S_3S^*_i$ and $\tilde{S}^*_jS_3 = S_3\tilde{S}^*_j$ for $1 \leqslant i, j \leqslant 2$. Let $G_1, 2G_2, 2\tilde{G}_1, \tilde{G}_2$ and $\hat{G}_1, 2\hat{G}_2, 2\hat{\tilde{G}}_1, \hat{\tilde{G}}_2$ be the fundamental operators of $\textbf{S}$ and $\textbf{S}^*$ respectively. Consider $\tilde{W} = (\tilde{W}_1, \tilde{W}_2)$ as above. Let $\tilde{\mathcal{L}} = \Ran \tilde{W}$. Then we have the following:
		\begin{enumerate}\label{Model 4}
			\item $S_1 \cong ((I \otimes G_1 + M^*_z \otimes \tilde{G}^*_2) \oplus (I \otimes \hat{G}^*_1 + M^*_{e^{it}} \otimes S_3S^*_3\hat{\tilde{G}}_2))|_{\tilde{\mathcal{L}}}$,
			
			\item $S_2 \cong ((I \otimes 2G_2 + M^*_z \otimes 2\tilde{G}^*_1) \oplus (I \otimes 2\hat{G}^*_2 + M^*_{e^{it}} \otimes 2S_3S^*_3\hat{\tilde{G}}_1))|_{\tilde{\mathcal{L}}}$,
			
			\item $S_3 \cong ((M^*_z \otimes I_{\mathcal{D}_{S_3}}) \oplus (M^*_{e^{it}} \otimes I_{\mathcal{D}_{S^*_3}}))|_{\tilde{\mathcal{L}}}$,
			
			\item $\tilde{S}_1 \cong ((I \otimes 2\tilde{G}_1 + M^*_z \otimes 2G^*_2) \oplus (I \otimes 2\hat{\tilde{G}}^*_1 + M^*_{e^{it}} \otimes 2S_3S^*_3\hat{G}_2))|_{\tilde{\mathcal{L}}}$,
			
			\item $\tilde{S}_2 \cong ((I \otimes \tilde{G}_2 + M^*_z \otimes G^*_1) \oplus (I \otimes \hat{\tilde{G}}^*_2 + M^*_{e^{it}} \otimes S_3S^*_3\hat{G}_1))|_{\tilde{\mathcal{L}}}$.
		\end{enumerate}
	\end{thm}
The following theorem gives the unitary invariance of a completely nonunitary $\Gamma_{E(3; 2; 1, 2)}$-contraction. The proof is similar to Theorem \ref{Thm 14}. Therefore, we omit the proof.

	\begin{thm}\label{Thm 16}
		Let $\textbf{S} = (S_1, S_2, S_3, \tilde{S}_1, \tilde{S}_2)$ and $\textbf{S}^{'} = (S^{'}_1, S^{'}_2, S^{'}_3, \tilde{S}^{'}_1, \tilde{S}^{'}_2)$ be two $\Gamma_{E(3; 2; 1, 2)}$-contractions on Hilbert spaces $\mathcal{H}$ and $\mathcal{H}^{'}$ respectively. Suppose $G_1, 2G_2, 2\tilde{G}_1, \tilde{G}_2$ be the fundamental operators of $\textbf{S}$ and $G^{'}_1, 2G^{'}_2, 2\tilde{G}^{'}_1, \tilde{G}^{'}_2$ be the fundamental operators of $\textbf{S}^{'}$ while $\hat{G}_1, 2\hat{G}_2, 2\hat{\tilde{G}}_1, \hat{\tilde{G}}_2$ be the fundamental operators of $\textbf{S}^*$ and $\hat{G}^{'}_1, 2\hat{G}^{'}_2, 2\hat{\tilde{G}}^{'}_1, \hat{\tilde{G}}^{'}_2$ be the fundamental operators of $\textbf{S}^{'*}$. Then $\textbf{S}$ is unitarily equivalent to $\textbf{S}^{'}$ if and only if the characteristic tuples of $\textbf{S}$ and $\textbf{S}^{'}$ are unitarily equivalent and the fundamental operators $\hat{G}_1, 2\hat{G}_2, 2\hat{\tilde{G}}_1, \hat{\tilde{G}}_2$ are unitarily equivalent to $\hat{G}^{'}_1, 2\hat{G}^{'}_2, 2\hat{\tilde{G}}^{'}_1, \hat{\tilde{G}}^{'}_2$ respectively.
	\end{thm}	
Let $(A,B,P)$ be a tetrablock contraction. Similarly, we can  define $\mathcal{A}^{'}, V^{'}, Q^{'}$ corresponding to $P$. The following proposition is the model for tetrablock contraction. As before, we can define $W^{\prime}=(W^{\prime}_1,W^{\prime}_2).$

	\begin{prop}\label{Prop tetra}
		Let $\textbf{T} = (A_1, A_2,P)$ be a tetrablock contraction on a Hilbert space $\mathcal{H}$. Let $F_1, F_2$ and $G_1, G_2$ be the fundamental operators of $\textbf{T}$ and $\textbf{T}^*$ respectively. Then
		\begin{enumerate}\label{Property tetra}
			\item $G^*_iD_P^*\mathcal{A}^{'1/2}|_{\overline{\Ran}\mathcal{A}^{'}} + PP^*G_{3-i}D_{P^*}\mathcal{A}^{'1/2}V^{'} = D_{P^*}A_i\mathcal{A}^{'1/2}|_{\overline{\Ran}\mathcal{A}^{'}}$,
			
			\item $G^*_iD_{P^*}P^* + PP^*G_{3-i}D_{P^*}
			= D_{P^*}A_iP^*$
			\end{enumerate}
		for $1 \leqslant i \leqslant 2$.
	\end{prop}
The following are model for completely non-unitary tetrablock contraction.
\begin{thm}[Model for special c.n.u  tetrablock contraction]\label{Model Thm tetra}
		Let $\textbf{T} = (A_1,A_2,P)$ be a c.n.u. tetrablock contraction on a Hilbert space $\mathcal{H}$ with $A^*_iP = PA^*_i$ for $1 \leqslant i \leqslant 2$. Let $F_1, F_2$ and $G_1, G_2$ be the fundamental operators of $\textbf{T}$ and $\textbf{T}^*,$ respectively. Consider $W^{'} = (W^{'}_1, W^{'}_2)$ as above and let $\mathcal{L}^{'} = \Ran W^{'}$. Then
		\begin{enumerate}\label{Model CNU tetra}
			\item $A_i \cong ((I \otimes F_i + M^*_z \otimes F^*_{3-i}) \oplus (I \otimes G^*_i + M^*_{e^{it}} \otimes PP^*G_{3-i}))|_{\mathcal{L}^{'}}$ for $1 \leqslant i \leqslant 2$,
			
			\item $P \cong ((M^*_z \otimes I_{\mathcal{D}_{P}}) \oplus (M^*_{e^{it}} \otimes I_{\mathcal{D}_{P^*}}))|_{\mathcal{L}^{'}}$.
		\end{enumerate}
	\end{thm}
		
Similarly, 	we describe the unitary invariance of a completely nonunitary tetrablock contraction.
\begin{thm}\label{Unitary Invariants tetra}
		Let $\textbf{T} = (A_1,A_2,P)$ and $\textbf{T}^{'} = (A^{'}_1, A^{'}_2, P^{'})$ be two tetrablock contractions on the Hilbert spaces $\mathcal{H}$ and $\mathcal{H}^{'},$ respectively. Suppose $F_1, F_2$ and $F^{'}_1, F^{'}_2$ are the fundamental operators for $\textbf{T}$ and $\textbf{T}^{'}$ respectively, and $G_1, G_2$ and $G^{'}_1, G^{'}_2$ are the fundamental operators for $\textbf{T}^*$ and $\textbf{T}^{'*},$ respectively. Then $\textbf{T}$ and $\textbf{T}^{'}$ are unitarily equivalent if and only if the characteristic tuples of $\textbf{T}$ and $\textbf{T}^{'}$ are unitarily equivalent and the fundamental operators $G_1, G_2$ are unitarily equivalent to $G^{'}_1, G^{'}_2$ respectively.
	\end{thm}

	\section{Counterexamples}\label{Section 5}
	
	In this section, we show that such abstract model of tetrablock contraction, $\Gamma_{E(3; 3; 1, 1, 1)}$-contraction and $\Gamma_{E(3; 2; 1, 2)}$-contraction may not exist if we drop the hypothesis of  $(\ref{Condition 3})$ $(\ref{Condition 1})$, and $(\ref{Condition 2}),$ respectively.

	\begin{exam}\label{Example 1}
		Let $\mathcal{H} = H^2(\mathbb{D}) = \{f \in \Hol(\mathbb{D}) : f(\zeta) = \sum_{n \geqslant 0} a_n\zeta^n, \sum_{n \geqslant 0} |a_n|^2 < \infty\}$ and $T_{\alpha}$ be an operator on $\mathcal{H}$ defined by
		\begin{equation}\label{T alpha}
			\begin{aligned}
				T_{\alpha}f(\zeta)
				&= \alpha a_0\zeta + a_1\zeta^2 + a_2\zeta^3 + \dots
			\end{aligned}
		\end{equation}
		where $\alpha \in \mathbb{D}$ and $f(\zeta) = \sum_{n \geqslant 0} a_n\zeta^n$, the power series expansion of $f$ around origin. It can be checked that
		\begin{equation}\label{T alpha star}
			\begin{aligned}
				T^*_{\alpha}f(\zeta)
				&= \overline{\alpha}a_1 + a_2\zeta + a_3\zeta^2 + \dots
			\end{aligned}
		\end{equation}
		and
		\begin{equation}\label{T alpha star square}
			\begin{aligned}
				T^2_{\alpha}f(\zeta)
				&= \alpha a_0\zeta^2 + a_1\zeta^3 + a_2\zeta^4 + \dots.
			\end{aligned}
		\end{equation}
		It is clear that $T_{\alpha}$ is a contraction. Then by Theorem 2.5 of \cite{Ball} we have that $(T_{\alpha}, T_{\alpha}, T^2_{\alpha})$ is a tetrablock contraction. Here $R_1 = R_2 = T_{\alpha}$ and $R_3 = T^2_{\alpha}$. Note that $R^*_1R_3 \ne R_3R^*_1$. Some routine computation shows that for $f(\zeta) = \sum_{n \geqslant 0} a_n\zeta^n$,
		\begin{equation}
			\begin{aligned}
				&D_{R_3}f(\zeta)
				= (1 - |\alpha|^2)^{1/2}a_0,\\
				&D_{R^*_3}f(\zeta)
				= a_0 + a_1\zeta + (1 - |\alpha|^2)^{1/2}a_2\zeta^2,\\
				&\mathcal{A}^{'1/2}f(\zeta)
				= |\alpha|a_0 + a_1\zeta + a_2\zeta^2 + \dots,\\
				&\mathcal{A}^{'1/2}_*f(\zeta) = 0,\\
				&Q^{'}f(\zeta) = f(\zeta),\\
				&\mathcal{H}^{'}_0 = \mathcal{H}.
			\end{aligned}
		\end{equation}
		It can also be checked that
		\begin{equation*}
			\begin{aligned}
				&R^*_1 - R_2R^*_3 = D_{R^*_3}G_1D_{R^*_3} \,\, \text{and} \,\, R^*_2 - R_1R^*_3 = D_{R^*_3}G_2D_{R^*_3},
			\end{aligned}
		\end{equation*}
		where $G_1f(\zeta) = G_2f(\zeta) = \overline{\alpha}a_1 + (1 - |\alpha|^2)^{1/2}a_2\zeta$ as $R_1 = R_2$.
		
		Then the constant term in $(I \otimes G^*_1 + M^*_{e^{it}} \otimes R_3R^*_3G_2)W^{'}_2$ is $D_{R^*_3}R_1\mathcal{A}^{'}$. Thus
		\begin{equation}
			\begin{aligned}
				D_{R^*_3}R_1\mathcal{A}^{'}f(\zeta)
				&= D_{R^*_3}R_1(|\alpha|^2a_0 + a_1\zeta + a_2\zeta^2 + \dots)\\
				&= D_{R^*_3}(\alpha|\alpha|^2a_0 + a_1\zeta + a_2\zeta^2 + \dots)\\
				&= \alpha|\alpha|^2a_0\zeta + (1 - |\alpha|^2)^{1/2}a_1\zeta^2,
			\end{aligned}
		\end{equation}
		and the constant term in $W^{'}_2R_1$ is $D_{R^*_3}\mathcal{A}^{'}R_1$. Thus, we have
		\begin{equation}
			\begin{aligned}
				D_{R^*_3}\mathcal{A}^{'}R_1f(\zeta)
				&= D_{R^*_3}\mathcal{A}^{'}(\alpha a_0\zeta + a_1\zeta^2 + a_2\zeta^3 + \dots)\\
				&= D_{R^*_3}(\alpha a_0\zeta + a_1\zeta^2 + a_2\zeta^3 + \dots)\\
				&= \alpha a_0\zeta + (1 - |\alpha|^2)^{1/2}a_1\zeta^2.
			\end{aligned}
		\end{equation}
		It is clear from here that the constant terms of $D_{R^*_3}R_1\mathcal{A}^{'}$ and $D_{R^*_3}\mathcal{A}^{'}R_1$ are not same.  This is a contradiction. Hence, the model described in Theorem \ref{Model Thm tetra} is not a c.n.u. tetrablock contraction.
	\end{exam}
	
	\begin{exam}\label{Example 2}
		Let $\mathcal{H}$ and $T_{\alpha}$ are as in Example \ref{Example 1}. Then wee have $(T_{\alpha}, 0, 0, 0, 0, T_{\alpha}, T^2_{\alpha})$ is a $\Gamma_{E(3; 3; 1, 1, 1)}$-contraction. In this example $T_1 = T_6 = T_{\alpha}, T_7 = T^2_{\alpha}$ and $T_2 = T_3 = T_4 = T_5 = 0$. It is easy to check that $T^*_1T_7 \ne T_7T^*_1$. It can be easily checked that $D_{T_7}, D_{T^*_7}, \mathcal{A}^{1/2}, \mathcal{A}^{1/2}_*, Q, \mathcal{H}_0$ are same as $D_{R_3}, D_{R^*_3}, \mathcal{A}^{'1/2}, \mathcal{A}^{'1/2}_*, Q^{'}, \mathcal{H}^{'}_0$ respectively.
We observe that
		\begin{equation*}
			\begin{aligned}
				&T^*_1 - T_6T^*_7 = D_{T^*_7}\tilde{F}_1D_{T^*_7} \,\, \text{and} \,\, T^*_6 - T_1T^*_7 = D_{T^*_7}\tilde{F}_6D_{T^*_7},
			\end{aligned}
		\end{equation*}
		where $\tilde{F}_1f(\zeta) = \tilde{F}_6f(\zeta) = \overline{\alpha}a_1 + (1 - |\alpha|^2)^{1/2}a_2\zeta$ as $T_1 = T_6$. It is important to note that the constant term in $(I \otimes \tilde{F}^*_1 + M^*_{e^{it}} \otimes T_7T^*_7G_2)W^{'}_2$ is $D_{T^*_7}T_1\mathcal{A}$. Thus, we have
		\begin{equation}
			\begin{aligned}
				D_{T^*_7}T_1\mathcal{A}f(\zeta)
				&= D_{T^*_7}T_1(|\alpha|^2a_0 + a_1\zeta + a_2\zeta^2 + \dots)\\
				&= D_{T^*_7}(\alpha|\alpha|^2a_0 + a_1\zeta + a_2\zeta^2 + \dots)\\
				&= \alpha|\alpha|^2a_0\zeta + (1 - |\alpha|^2)^{1/2}a_1\zeta^2,
			\end{aligned}
		\end{equation}
Also, the constant term in $W_2T_1$ is $D_{T^*_7}\mathcal{A}T_1$. Hence, we get
		\begin{equation}
			\begin{aligned}
				D_{T^*_7}\mathcal{A}T_1f(\zeta)
				&= D_{T^*_7}\mathcal{A}(\alpha a_0\zeta + a_1\zeta^2 + a_2\zeta^3 + \dots)\\
				&= D_{T^*_7}(\alpha a_0\zeta + a_1\zeta^2 + a_2\zeta^3 + \dots)\\
				&= \alpha a_0\zeta + (1 - |\alpha|^2)^{1/2}a_1\zeta^2.
			\end{aligned}
		\end{equation}
		It is clear from here that the constant terms of $D_{T^*_7}T_1\mathcal{A}^{'}$ and $D_{T^*_7}\mathcal{A}T_1$ are not same. This leads  to a contradiction. Hence, the model described in Theorem \ref{Model Thm tetra} is not a c.n.u. $\Gamma_{E(3; 3; 1, 1, 1)}$-contraction.
	\end{exam}
	
	We use Example \ref{Example 2} to find a similar example of c.n.u. $\Gamma_{E(3; 2; 1, 2)}$-contraction that does not satisfy $(
	\ref{Condition 2})$.
	
	\begin{exam}\label{Example 3}
		Let $\mathcal{H}$ and $T_{\alpha}$ are as in Example \ref{Example 1}. Then $(T_{\alpha}, 0, 0, 0, 0, T_{\alpha}, T^2_{\alpha})$ is a $\Gamma_{E(3; 3; 1, 1, 1)}$-contraction. By Proposition 2.10 of \cite{apal2} we have that $(T_{\alpha}, 0, T^2_{\alpha}, 0, T_{\alpha})$ is a $\Gamma_{E(3; 2; 1, 2)}$-contraction. In this example $S_1 = \tilde{S}_2 = T_{\alpha}, S_2 = \tilde{S}_1 = 0$ and $S_3 = T^2_{\alpha}$. It is easy to check that $S^*_1S_3 \ne S_3S^*_1$. Some routine computation show $D_{S_3}, D_{S^*_3}, \tilde{\mathcal{A}}^{1/2}, \tilde{\mathcal{A}}^{1/2}_*, \tilde{Q}, \tilde{\mathcal{H}}_0$ are same as $D_{T_7}, D_{T^*_7}, \mathcal{A}^{1/2}, \mathcal{A}^{1/2}_*, Q, \mathcal{H}_0$ respectively. It can also be checked that
		\begin{equation*}
			\begin{aligned}
				&S^*_1 - \tilde{S}_2S^*_3 = D_{S^*_3}\hat{G}_1D_{S^*_3} \,\, \text{and} \,\, \tilde{S}^*_2 - S_1S^*_3 = D_{S^*_3}\hat{\tilde{G}}_2D_{S^*_3},
			\end{aligned}
		\end{equation*}
		where $\hat{G}_1f(\zeta) = \hat{\tilde{G}}_2f(\zeta) = \overline{\alpha}a_1 + (1 - |\alpha|^2)^{1/2}a_2\zeta$ as $S_1 = \tilde{S}_2$.
		
		Note that the constant term in $(I \otimes \hat{G}^*_1 + M^*_{e^{it}} \otimes T_7T^*_7\hat{\tilde{G}}_2)W_2$ is $D_{S^*_3}S_1\mathcal{A}$. Thus, we get
		\begin{equation}
			\begin{aligned}
				D_{S^*_3}S_1\mathcal{A}f(\zeta)
				&= D_{S^*_3}S_1(|\alpha|^2a_0 + a_1\zeta + a_2\zeta^2 + \dots)\\
				&= D_{S^*_3}(\alpha|\alpha|^2a_0 + a_1\zeta + a_2\zeta^2 + \dots)\\
				&= \alpha|\alpha|^2a_0\zeta + (1 - |\alpha|^2)^{1/2}a_1\zeta^2,
			\end{aligned}
		\end{equation}
Also, the constant term in $W_2S_1$ is $D_{S^*_3}\mathcal{A}S_1$. Thus, we have
		\begin{equation}
			\begin{aligned}
				D_{S^*_3}\mathcal{A}S_1f(\zeta)
				&= D_{S^*_3}\mathcal{A}(\alpha a_0\zeta + a_1\zeta^2 + a_2\zeta^3 + \dots)\\
				&= D_{S^*_3}(\alpha a_0\zeta + a_1\zeta^2 + a_2\zeta^3 + \dots)\\
				&= \alpha a_0\zeta + (1 - |\alpha|^2)^{1/2}a_1\zeta^2.
			\end{aligned}
		\end{equation}
This shows  that the constant terms of $D_{S^*_3}S_1\mathcal{A}$ and $D_{S^*_3}\mathcal{A}S_1$ are not equal, which leads to a contradiction. Hence, the model described in Theorem \ref{Thm 11} is not a c.n.u. $\Gamma_{E(3; 2; 1, 2)}$-contraction.
	\end{exam}

\textsl{Acknowledgements:}
The authors gratefully acknowledge the support of the Council of Scientific and Industrial Research (CSIR), Government of India, for the research of the second author under File No:  09/1002(17194)/ 2023-$\rm{EMR-I.}$ The third-named author is supported by the research project of SERB with ANRF File Number: CRG/2022/003058, by the Science and Engineering Research Board (SERB), Department of Science and Technology (DST), Government of India.

\end{document}